\documentclass[12pt]{amsart}

\usepackage[all]{xypic}
\usepackage{tikz}
\usepackage{ragged2e}
\usetikzlibrary{arrows} 
\usetikzlibrary{decorations.markings}
\usepackage{graphicx}
\usepackage{bm}
\usepackage{float}
\usepackage{epsf}
\usepackage{verbatim} 
\usepackage{amsmath}
\usepackage{amsfonts}
\usepackage{amssymb}
\usepackage{mathrsfs}
\usepackage{amsthm}
\usepackage{newlfont}
\usepackage{enumitem}
\usepackage[new]{old-arrows}
\usepackage{booktabs}
\usepackage{enumitem}
\usepackage{makecell}
\usepackage[multiple]{footmisc}
\usepackage[answerdelayed,lastexercise]{exercise}
\usepackage[hidelinks]{hyperref}
\usepackage{mleftright}

\usepackage{fnpct}
\usepackage{colonequals} 
\usepackage{stmaryrd}

\usepackage{apptools}
\usepackage{chngcntr}

\usepackage[margin=1in]{geometry}

\newtheorem{thm}{Theorem}[section]
\newtheorem{prop}[thm]{Proposition}
\newtheorem{lem}[thm]{Lemma}
\newtheorem{cor}[thm]{Corollary}

\theoremstyle{definition}
\newtheorem{defn}[thm]{Definition}

\newtheorem{example}[thm]{Example}

\theoremstyle{remark}
\newtheorem{rem}[thm]{Remark}

\numberwithin{equation}{section}
\numberwithin{thm}{section}

\makeatletter 
\newcommand\mynobreakpar{\vspace{0.02in}\par\nobreak\@afterheading}  
\makeatother

\makeatletter
\@addtoreset{footnote}{section}
\makeatother

\usepackage{xcolor}
\hypersetup{
	colorlinks,
	linkcolor={red!50!black},
	citecolor={blue!50!black},
	urlcolor={blue!80!black}
}

\AtBeginDocument{%
	\def\MR#1{}
}

\DeclareMathOperator{\dist}{dist}

\DeclareMathOperator{\End}{End}
\DeclareMathOperator{\Aut}{Aut}
\DeclareMathOperator{\rank}{rank}

\DeclareMathOperator{\Ht}{ht}
\DeclareMathOperator{\Frob}{Frob}

\DeclareMathOperator{\chr}{char}

\DeclareMathOperator{\Mat}{Mat}
\DeclareMathOperator{\ord}{ord}

\DeclareMathOperator{\Gal}{Gal}

\DeclareMathOperator{\Tr}{Tr}

\DeclareMathOperator{\diag}{diag}
 
\DeclareMathOperator{\Nr}{Nr}

\DeclareMathOperator{\Cent}{Cent}

\DeclareMathOperator{\lcm}{lcm}

\newcommand{\id}{\mathrm{id}}

\newcommand{\modr}{\ \mathrm{mod}_r\ }

\newcommand{\ev}{\mathrm{ev}}

\newcommand{\fl}{\mathfrak{l}}

\newcommand{\fn}{\mathfrak{n}}

\newcommand{\fp}{\mathfrak{p}}
\newcommand{\fq}{\mathfrak{q}}

\newcommand{\cC}{\mathcal{C}}

\newcommand{\cF}{\mathcal{F}}

\newcommand{\cM}{\mathcal{M}}

\newcommand{\cP}{\mathcal{P}}

\renewcommand{\cR}{\mathcal{R}}
\newcommand{\cS}{\mathcal{S}}

\newcommand{\F}{\mathbb{F}}

\newcommand{\Z}{\mathbb{Z}}

\newcommand{\tF}{\widetilde{F}}

\newcommand{\tinf}{\widetilde{\infty}}

\newcommand{\To}{\longrightarrow}

\newcommand{\Fi}{F_\infty}

\newcommand{\oK}{\overline{K}}

\newcommand{\twist}[1]{#1\!\left\{\tau\right\}}

\newcommand{\twistt}[2]{#1\!\left\{#2\right\}}
\newcommand{\atwist}[2]{#1\langle#2\rangle}

\newcommand{\abs}[1]{\left|#1\right|}

\title{Semifields and rank-metric codes from Drinfeld modules}

\author{Giacomo Micheli}
\address{Department of Mathematics \& Statistics, The University of South Florida, Tampa, Florida, United States of America}
\email{gmicheli@usf.edu}

\author{Mihran Papikian}
\address{Department of Mathematics, Pennsylvania State University, University Park, Pennsylvania, United States of America}
\email{papikian@psu.edu}

\thanks{The first author was supported by NSF CAREER grant 2338424. The second author was supported in part by the Simons Foundation, award number MPS-TSM-00008093.}

\subjclass[2020]{11G09, 12K10, 11T71, 16S36}
\keywords{Rank metric codes, Drinfeld modules, twisted polynomial ring}

\begin{document}

\begin{abstract}
We establish a connection between Drinfeld modules and rank-metric codes, focusing  
on maximum rank distance codes arising from finite semifields. 
Our method constructs rank-metric codes from linear subspaces of endomorphisms of a 
Drinfeld module acting on torsion submodules. Using a duality between a Drinfeld module and its endomorphisms, 
we show that Sheekey's construction \cite{Sheekey} fits naturally into this framework, yielding a short
conceptual proof of his main theorem. 
Next, to demonstrate the potential of our method, we construct, for every prime power $q$, new semifields of order $q^9$,
which are distinguished from the known ones by their nuclear parameters. 
Finally, going beyond the Drinfeld module $\phi_T=\tau^n$ that underlies Sheekey's and related constructions, 
we obtain infinite families of semifield codes from Drinfeld modules with smaller endomorphism rings, 
and show that the corresponding semifields are cyclic semifields of Petit.
\end{abstract}

\maketitle
\section{Introduction} 

\subsection{Rank-metric codes and semifields}\label{ss1.1} Error-correcting codes are used when communication is over a channel in which 
errors may occur. This requires a set equipped with a distance function and a subset of allowed codewords; if the number of errors 
is assumed to be small, then a received message is decoded to be the nearest valid codeword. 

The most well-known and widely-used examples are the codes with the Hamming metric: the codewords are taken from an $n$-dimensional vector 
space $\F_q^n$ over a finite field $\F_q$ with $q$ elements (usually $q=2$), and the distance between two vectors in $\F_q^n$ 
is defined as the number of positions in which they differ.  

In rank-metric coding, the codewords are instead taken from the set $\Mat_{m, n}(\F_q)$  of $m\times n$-matrices ($m\leq n$), with 
the distance between two matrices defined as the rank of their difference 
$$
\dist(X, Y) = \rank (X-Y). 
$$
Rank-metric codes have seen renewed interest in recent years due to their applications in network coding, distributed
data storage,  
and cryptography; cf. \cite{SKK}, \cite{bartz2022}, \cite{SheekeySurvey}. 

A rank-metric code is a subset $\cC \subset \Mat_{m, n}(\F_q)$ with a minimum distance defined by 
$$
\dist(\cC)=\min \{\dist(X, Y)\mid X, Y\in \cC, X\neq Y\}. 
$$
Such $\cC$ satisfies the  
 \textit{Singleton-type bound}: if $d=\dist(\cC)$, then 
\begin{equation}\label{eqSingleton}
	\abs{\cC} \leq q^{n(m-d+1)}. 
\end{equation}
A code attaining the bound \eqref{eqSingleton} is said to be a \textit{Maximum Rank Distance code}, or \textit{MRD code} for short. It is known that 
MRD codes exist for every finite field $\F_q$ and any choice of parameters $d\leq m\leq n$; see  \cite{Delsarte}. An MRD code 
can correct up to $\lfloor (d-1)/2\rfloor$ errors.

Unless stated otherwise, we will assume for the rest of this article that MRD codes that we consider are $\F_q$-linear, i.e., $\cC$
is an $\F_q$-vector subspace of $\Mat_{m, n}(\F_q)$ (in Theorem \ref{thmSheekey} we also allow $\F_{q_0}$-linear codes for a subfield $\F_{q_0}\subseteq \F_q$). 
In the special case when $d=m=n$, an $\F_q$-linear MRD code is equivalent to a 
\textit{semifield}, i.e., a finite dimensional $\F_q$-algebra with unity and without zero divisors, in which multiplication is not 
necessarily commutative or associative; see \cite[Prop. 1]{Sheekey}.  
To distinguish such MRD codes, 
we will call them \textit{semifield codes}, although this is not a standard terminology. 

To see the connection between semifields and MRD codes, let $\cS$ be a semifield of dimension $n$ over $\F_q$.  
Then the left multiplication by a non-zero element $\alpha$ of $\cS$ defines an 
invertible linear operator $T_\alpha\colon \cS \to \cS$, $\beta\mapsto \alpha\beta$. After fixing a basis of $\cS$ over $\F_q$, we get 
an injection $\cS\to \Mat_n(\F_q)$, whose image is a semifield code. 
For example, this can be applied to $\cS=k$, where 
$$k\colonequals \F_{q^n}$$ is the finite field with $q^n$ elements, considered as an $\F_q$-vector space.  

\begin{example}\label{exPetit}
A more interesting example of a finite semifield arises from the \textit{twisted polynomial ring} 
$$
R\colonequals \twist{k}=\left\{\sum_{i=0}^m a_i \tau^i \mid m\geq 0, a_i\in k\right\}, 
$$ 
where the addition is the usual addition of polynomials, but 
the multiplication is subject to the following commutation rule $$\tau a=a^q\tau \quad \text{for all} \quad a\in k.$$ It is clear that $R$ is a (noncommutative) ring without zero-divisors. It is also easy to show that $R$ possesses a right (and left) division algorithm. 
Let $P\in R$ be a monic  irreducible element of degree $s$. For any $f\in R$, there are unique $g, e\in R$ such that 
$$
f=gP+e \quad \text{with } \deg(e)<s \text{ or }e=0. 
$$
We denote the residue $e=f\modr P$. With this notation, the set 
$$
\cS=\{f\in R\mid \deg(f)\leq s-1\},  
$$
with the addition from $R$ and multiplication 
$a\circ b = ab\modr P$ 
is a semifield of order $q^{ns}$. This construction is due to Petit; the semifields thus obtained are called \textit{cyclic semifields} (or Jha--Johnson semifields); cf. \cite{Sheekey}.   
\end{example}

Elliptic curves have enjoyed great success in cryptographic applications. 
Drinfeld modules have many formal similarities to elliptic curves, but the direct analogues for Drinfeld modules of the 
cryptosystems based on elliptic curves are insecure, because the linear structure of Drinfeld modules reduces the underlying 
problems to linear algebra \cite{Scanlon}. This same linear structure turns out to be ideally suited for constructing rank-metric codes. 
In this paper, we explain how the theory of Drinfeld modules can be used to construct rank-metric codes. 
The Drinfeld module perspective may situate various constructions of such codes, 
e.g. \cite{Sheekey}, \cite{TZ},  \cite{LSS}, \cite{GTLNS}, within a unified algebraic framework, 
opening the door to powerful tools from the theory of Drinfeld modules into coding theory.
We note that Drinfeld modules were already used in \cite{BDM} to construct optimal rank-metric codes with rank-locality.


\subsection{Drinfeld modules}\label{sIntroDM} 
We denote $A=\F_q[T]$, the polynomial ring in indeterminate $T$ with coefficients in $\F_q$. 
Given a nonzero ideal $\fn\lhd A$, with slight abuse of notation, we denote by the same symbol  
the unique monic generator of $\fn$. The \textit{primes} of $A$ are the maximal ideals of $A$. 
Given a prime $\fp\lhd A$, we denote $\F_\fp=A/\fp$. 

A \textit{Drinfeld module} of rank $r\geq 1$ over $k$ is an $\F_q$-algebra homomorphism 
\[\phi\colon A\to R,\quad  a\mapsto \phi_a, \] 
such that $\phi_T=g_0+g_1\tau+\cdots+g_r\tau^r$ for some $g_0, \dots, g_r\in k$ with $g_r\neq 0$. The \textit{$A$-characteristic of $\phi$}, 
denoted $\chr_A(\phi)$, is the kernel of the homomorphism $A\to k$ defined by $T\mapsto g_0$. The Drinfeld module $\phi$ makes the algebraic 
closure $\bar{k}$ into an $A$-module, denoted ${^\phi}\overline{k}$; see Section \ref{sDM} for the details. The elements of ${^\phi}\overline{k}$ 
annihilated by $a\in A$ form a finite $A$-module denoted $\phi[a]$; if $\chr_A(\phi)\nmid a$, then $\phi[a]\cong (A/aA)^r$. 

Define 
the \textit{endomorphism ring of $\phi$} as 
$$
\End(\phi)=\left\{u\in R\mid u\phi_T=\phi_Tu\right\}. 
$$
$\End(\phi)$ is an $A$-algebra, which is known to be a free $A$-module of rank $\leq r^2$. Every $u\in \End(\phi)$ acts $A$-linearly on $\phi[a]$. 
Choosing a prime $\fp\neq \chr_A(\phi)$, we obtain a homomorphism 
$$
\iota_\fp\colon \End(\phi)\To \End_A(\phi[\fp])\cong \Mat_r(\F_\fp). 
$$

\medskip
\noindent 
\textbf{\underline{The recipe}.} Our construction of rank-metric codes relies on three main ingredients:
\begin{itemize}
\item A Drinfeld module $\phi$ of rank $r$.
\item An $\F_q$-linear subspace $\cM \subset \End(\phi)$.
\item A prime $\fp \in A$, distinct from the $A$-characteristic of $\phi$, such that every nonzero $u \in \cM$ acts \textit{invertibly} on $\phi[\fp]$, i.e., as an invertible linear transformation. 
\end{itemize}
Given this, we obtain an embedding
\[
\iota_\fp \colon \cM \longhookrightarrow \End_{\F_\fp}(\phi[\fp]) \cong \Mat_r(\F_\fp), 
\]
where each nonzero $\iota_\fp(u)$ has full rank. If additionally 
\[\dim_{\F_q} \cM = r \cdot \deg(\fp),\] then the image $\iota_\fp(\cM)$ is a semifield code. 
Here, and throughout the paper, an $\F_q$-subspace of $\End_{\F_\fp}(\phi[\fp])\cong \Mat_r(\F_\fp)$ is regarded as a code in 
$\End_{\F_q}(\phi[\fp])\cong \Mat_{r\deg(\fp)}(\F_q)$ by restriction of scalars. By construction, the elements of such a code
are $\F_\fp$-linear maps of $\phi[\fp]$, so the code lies in the subalgebra $\Mat_r(\F_\fp)$ of $\Mat_{r\deg(\fp)}(\F_q)$.
The code itself is an $\F_q$-subspace, but in general not an $\F_\fp$-subspace; for instance, the semifield codes of
Section \ref{secHigher} have center $\F_q$.

\medskip

The flexibility of choosing $\phi$ and $\cM$ makes this construction very general. 
The advantage of this perspective is that it provides access to the rich theory of 
Drinfeld modules. The central challenge is to find a suitable $\fp$ satisfying the invertibility condition; here one can leverage various tools from the theory of Drinfeld modules, such as the equality between the $\tau$-degree of an endomorphism and the $T$-degree of its determinant on the associated Tate module. We focus on the semifield case, though this strategy extends to MRD codes by relaxing the invertibility requirement.

\begin{example}\label{BasicExample}
We demonstrate the idea with a simple example. 
Define $\phi$ by $\phi_T=\tau^n$. The rank of $\phi$ is $n$, $\chr_A(\phi)=T$, and $\End(\phi)=R$. 
Let $\fp\neq T$ be a prime of degree $d$ and let
$$
\cM=\{u\in R\mid \deg_\tau(u)\leq d-1\}. 
$$
We claim that every $0\neq u\in \cM$ acts invertibly on $\phi[\fp]$. We argue by contradiction. Assume $0\neq \beta\in \phi[\fp]$ 
is annihilated by $u$. Then the $A$-submodule of $\phi[\fp]$ generated by $\beta$ has size $q^d$, yet $\deg_\tau(u)<d$ implies $\#\ker(u)\leq q^{d-1}$, a contradiction. 
Since $\dim_{\F_q}\cM=nd=\deg(\fp)\cdot n$, the image $\iota_\fp(\cM)$ is a semifield code.
\end{example}

\subsection{Main results} 
In Section \ref{secSheekey}, we give a reinterpretation of Sheekey's construction
\cite{Sheekey} in terms of Drinfeld modules and a short proof
of one of the main results in \cite{Sheekey}.
Sheekey's construction has been generalized in \cite{TZ} and \cite{LSS} using the
methods in \cite{Sheekey}, which rely on direct computations in skew-polynomial rings
and their quotients.

In our Drinfeld module reinterpretation of Sheekey's construction, as in Example \ref{BasicExample}, one takes $\phi_T=\tau^n$.
Let $\fp\neq T$ be a monic prime in $A$ of degree $d$ with constant coefficient $\fp_0$, let $1\leq m\leq n$, $\eta\in k$ and $\rho\in \Aut(k)$, 
and put
$$
\cM_m(\eta, \rho) \colonequals \{u_0+u_1\tau+\cdots+u_{dm}\tau^{dm}\in R\mid  u_{dm}=\eta\rho(u_0)\}.
$$
Let $\F_{q_0}$ be the subfield of $\F_q$ fixed by $\rho$. We show that if 
\[\Nr_{k/\F_{q_0}}(\eta)\cdot \Nr_{\F_q/\F_{q_0}}\left((-1)^{dm(n-1)}\fp_0^m\right)\neq 1,\] then the image of $\cM_m(\eta, \rho)$ in 
$\End_{\F_\fp}(\phi[\fp])\cong \Mat_n(\F_\fp)$ is an MRD code with minimum distance $n-m+1$, and for $m=1$ and $\rho$ fixing $\F_q$ 
it is a semifield code (Theorem \ref{thmSheekey}). This recovers \cite[Thm. 7]{Sheekey} for the Frobenius automorphism. 
Note that choosing $m=1$ and $\eta=0$ gives Example \ref{BasicExample}. The proof
uses the duality between $u \in \cM_m(\eta, \rho)$
being an endomorphism of $\phi$ and $\phi_T$ being an endomorphism of the auxiliary
Drinfeld module $\psi$ defined by $\psi_x=u$: this converts invertibility questions into questions about
characteristic polynomials of endomorphisms, for which standard tools from the theory of Drinfeld
modules are available. This duality is made precise in Section \ref{secDuality}: the characteristic polynomial of $u$ acting on the
Tate module of $\phi$ and the characteristic polynomial of $\phi_T$ acting on the Tate module of $\psi$ are
the same polynomial in two variables, up to a factor in $\F_q^\times$ (Theorem \ref{thmDuality}). In particular, if $u$ has nonzero 
constant term, then the determinant of $u$ is, up to a scalar,
the characteristic polynomial of $\phi_T$ acting on $\ker(u)$, which gives a simple criterion for $u$ to act
invertibly on $\phi[\fp]$ (Corollary \ref{corInv}). For $\phi_T=\tau^n$, the constant term of this characteristic polynomial is 
determined by the norm of $u_0/u_s$, where $s=\deg_\tau(u)$ (Proposition \ref{propNorm}), and the key point in the proof of 
Theorem \ref{thmSheekey} is that the Frobenius endomorphism $\pi=\tau^n$ acts on $\phi[\fp]$ as a scalar.
The interesting feature of Sheekey's construction was that the semifield codes it produced in some cases corresponded to 
genuinely new families of semifields, with parameters for which no examples had been known before.

In Section \ref{secHigher} we keep $\phi_T=\tau^n$, but allow $\cM$ to contain endomorphisms of degree larger than $\deg(\fp)$.
We consider the message spaces $\cM_V=k\cdot V$, where $V\subset \F_q\{\tau\}$ is an $\F_q$-subspace of dimension $n$
containing $1$; these are stable under left multiplication by $k$, but in general not under right multiplication, so their left
and right idealizers are different (Lemma \ref{lemNucHigher}). For $n=\deg(\fp)=3$ and every prime power $q$, we construct explicit 
semifield codes of this type with nuclear parameters $(9; 3, 1, 3; 1)$ (Theorem \ref{thmq9} and Proposition \ref{propNucq9}), i.e., 
semifields of order $q^9$ which are $3$-dimensional over their left and right nuclei, with middle nucleus and center $\F_q$. 
The proof uses the traces of the Frobenius and of its inverse on the kernel of an endomorphism, which can be read off from the 
two ends of the determinant of that endomorphism. Exhaustive computations for $q=2, 3$ show that the family $\cM_V$ contains 
many more such codes (Example \ref{exComp}). 
To the best of our knowledge, no semifields with these nuclear parameters, or with nuclear parameters in the same Knuth orbit,
were previously known; see $\S$\ref{ssNovelty} for a comparison with the known families. We regard this as a proof of concept
for the Drinfeld module approach. 

In Section \ref{s2C}, we extend the setup as follows: 
Let $\ell$ and $s$ be integers such that $n=\lcm(\ell, s)$. We choose $\phi$ so that $\phi_T\in \twistt{\F_{q^s}}{\tau^\ell}\subset \twist{k}$ 
and $\deg(\phi_T)=r\ell$. In this case, $\twistt{\F_{q^\ell}}{\tau^s}\subset \End(\phi)$. Let $\fp$ be a prime different from 
$\chr_A(\phi)$ and denote $d=\deg_T(\fp)$. We take 
$$
\cM=\left\{\sum_{i=0}^{rd-1} b_i\tau^{s i}\mid b_i\in \F_{q^\ell}\right\}. 
$$
Unlike the constructions in \cite{Sheekey, TZ, LSS, GTLNS}, which all take $\phi_T = \tau^n$, 
our construction uses Drinfeld modules with smaller endomorphism rings, exploiting 
the structure of division subalgebras of $k(\tau)$.
We show that the image of $\cM$ in $\End_{\F_\fp}(\phi[\fp])$ is a semifield code if the minimal polynomial over $\F_{q^g}$, $g=\gcd(\ell, s)$, of 
the Frobenius endomorphism $\pi=\tau^n\in \End(\phi)$ acting on $\phi[\fp]$ is irreducible of degree $rd$ (Proposition \ref{propCriterion}). 
The proof uses the duality of Section \ref{secDuality}, applied to $u\in \cM$ and $\pi$. This condition holds under some technical assumptions 
on the reduction modulo $\fp$ of the minimal polynomial $m_\phi(x)$ of $\pi$ over $\F_q(\phi_T)$ (Theorem \ref{thmFamilies}). Next, 
using the Chebotarev density theorem, we show that, when $\gcd(\ell, s)=1$, one can choose $\phi$ and $\fp$ so that the assumptions on $m_\phi(x)$ 
are satisfied (Theorem \ref{thmExistence}). Finally, we show that the semifields obtained in this way are cyclic semifields of Petit, 
defined over the twisted polynomial ring $\twistt{\F_{q^\ell}}{\tau^s}$, and we deduce their nuclear parameters
(Proposition \ref{propPetit} and Corollary \ref{corNP6}). Thus, the construction of Section \ref{s2C} does not produce new semifields. 
Rather, it gives an arithmetic realization of cyclic semifields, whose isotopy classes were recently classified in \cite{SZ}: 
the isotopy class is determined by the action of the Frobenius of a Drinfeld module with a small endomorphism ring on its $\fp$-torsion. 
The examples of Section \ref{secHigher} show that the Drinfeld module approach does lead to semifields 
with new nuclear parameters.


\section{Preliminaries on Drinfeld modules}\label{sDM}  In this section we recall some basic facts about Drinfeld modules 
and collect several auxiliary results that will be used later in the paper. Proofs of the standard facts can be found in 
\cite{PapikianGTM}.

\subsection{Definition} Let $F=\F_q(T)$ be the fraction field of $A=\F_q[T]$. For $0\neq a\in A$, let $\deg(a)\in \Z_{\geq 0}$ be the usual degree 
of $a$ as a polynomial in $T$. 

Let $K$ be a field containing $\F_q$ as a subfield. Let 
$\twist{K}$ be the twisted polynomial ring already discussed in $\S$\ref{ss1.1} in the case when $K=\F_{q^n}$ (so the commutation rule is $\tau a=a^q\tau$ for all $a\in K$). 
There is a homomorphism 
\[ \partial\colon \twist{K}\to K, \qquad \sum_{i=0}^n a_i \tau^i \mapsto a_0,\] called the \textit{derivative}. 
For $f=a_h\tau^h+\cdots +a_n\tau^n$ with $0\leq h\leq n$, $a_h\neq 0$ and $a_n\neq 0$, we define the \textit{height} of $f$ as $\Ht(f)=h$, and the \textit{degree} of $f$ as $\deg_\tau(f)=n$. 
The map $$f=\sum_{i=0}^n a_i \tau^i \longmapsto f(x)=\sum_{i=0}^n a_i x^{q^i}$$ gives a ring isomorphism between $\twist{K}$ 
and the ring $\atwist{K}{x}$ of $\F_q$-linear polynomials, where on $\atwist{K}{x}$ the multiplication is defined via the substitution $f_1\circ f_2=f_1(f_2(x))$. 


A \textit{Drinfeld module} of rank $r\geq 1$ over $K$ is an $\F_q$-algebra homomorphism 
\begin{align*}
	\phi\colon A &\To \twist{K}, \\ 
	a &\longmapsto \phi_a=g_0(a)+g_1(a)\tau+\cdots +g_n(a)\tau^n,
\end{align*}
such that for $a\neq 0$ we have $n=\deg(a) r$ and $g_n(a)\neq 0$. Note that $\phi$ is always injective, so gives an embedding of $A$ 
into the non-commutative ring $\twist{K}$. Moreover, $\phi$ is uniquely determined by $\phi_T$, so to define a Drinfeld module over $K$ 
one simply needs to choose $g_0, g_1, \dots, g_r\in K$ such that $g_r\neq 0$ and put $\phi_T=g_0+g_1\tau+\cdots+g_r\tau^r$. 
We define a homomorphism $\gamma\colon A\to K$ via $a\mapsto \partial \phi_a$. The \textit{$A$-characteristic of $\phi$} is 
$\chr_A(\phi)\colonequals \ker(\gamma)$; note that $\chr_A(\phi)$ is either $0$ or a prime of $A$.
If $\fl=\chr_A(\phi)\neq 0$, the \textit{height} of $\phi$ is $H(\phi)\colonequals\Ht(\phi_\fl)/\deg(\fl)$.
By \cite[Lem. 3.2.11]{PapikianGTM}, $H(\phi)$ is an integer with $1\leq H(\phi)\leq r$. We say that $\phi$ is
\textit{ordinary} if $H(\phi)=1$.

\subsection{Torsion module and endomorphisms}\label{secTM} 
Let $\phi\colon A\to \twist{K}$ be a Drinfeld module over $K$ of rank $r$. Through $\phi$, the algebraic closure $\oK$ of $K$ acquires an $A$-module structure, where $a\in A$ acts on $\beta\in \oK$ by $a\ast \beta=\phi_a(\beta)$, i.e., we substitute $\beta$ into the 
polynomial $\phi_a(x)$. Denote this module by ${^\phi}{\oK}$. One is interested in its torsion submodule. More precisely, given $0\neq a\in A$, 
the \textit{$a$-torsion submodule} of ${^\phi}{\oK}$, denoted $\phi[a]$,  is the set of roots of the polynomial $\phi_a(x)$. It is easy to check that $\phi[a]$ 
is indeed an $A$-submodule of ${^\phi}{\oK}$, i.e., is invariant under the action of any $b\in A$, where $b$ acts via $\phi_b$. 
If $\chr_A(\phi)$ does not divide $a$, then $\phi[a]\cong (A/aA)^r$; see \cite[Cor. 3.5.3]{PapikianGTM}. 

Let $\fp\lhd A$ be a prime different from $\chr_A(\phi)$. The finite $A$-modules $\phi[\fp^n]\cong (A/\fp^n)^r$, $n\geq 1$, form a projective system 
with transition maps $\phi[\fp^{n+1}]\to \phi[\fp^n]$, $\alpha\mapsto \phi_\fp(\alpha)$. Taking the inverse limit, one obtains the \textit{$\fp$-adic Tate module} 
$T_\fp(\phi)\cong A_\fp^r$ of $\phi$, where $A_\fp$ is the ring of integers of the completion of $F$ with respect to the absolute value arising from the $\fp$-adic valuation on $F$. 

The \textit{endomorphisms of $\phi$} defined over $K$ are the elements of $\twist{K}$ which induce endomorphisms of ${^\phi}\oK$. More explicitly,
\begin{align*}
\End(\phi) &=\{u\in \twist{K}\mid u\phi_a=\phi_au\text{ for all }a\in A\} \\ 
& = \{u\in \twist{K}\mid u\phi_T=\phi_Tu\}. 
\end{align*}
Obviously, the image $\phi(A)$ of $A$ under $\phi$ is in $\End(\phi)$. In fact, it is easy to check that $\End(\phi)$ is an $A$-algebra, where we 
identify $\phi(A)$ with $A$. Moreover, one can show 
that $\End(\phi)$ is a free $A$-module of rank $\leq r^2$; see \cite[Thm. 3.4.1]{PapikianGTM}. 

Any $u\in \End(\phi)$ induces an $A$-linear transformation on $\phi[a]$, and also an $A_\fp$-linear transformation on $T_\fp(\phi)$. 
Let $P_{\phi, u}(x)\in A_\fp[x]$ denote the characteristic polynomial of this latter transformation. It turns out that the coefficients of $P_{\phi, u}(x)$ are 
actually in $A$ and do not depend on $\fp$; see \cite[Thm. 3.6.6]{PapikianGTM}. The characteristic polynomial of $u$ acting on the $\F_\fp$-vector space $\phi[\fp]\cong \F_\fp^r$ is the reduction modulo $\fp$ of $P_{\phi, u}(x)$. In particular, $u$ acts \textit{invertibly} on $\phi[\fp]$ 
 (i.e., as an invertible linear transformation) if and only if $P_{\phi, u}(0)$ is not divisible by $\fp$.  
 
 One of the key technical questions arising in this 
 paper is to show that certain endomorphisms act invertibly on $\phi[\fp]$ for an appropriately chosen $\fp$. More precisely, 
 we choose a specific $\F_q$-vector subspace $\cM\subset \End(\phi)$--\textit{the message space}--and show that every $0\neq u\in \cM$
acts invertibly on $\phi[\fp]$. In that situation, $\cM$ defines a semifield code in $\End_{\F_\fp}(\phi[\fp])\cong \Mat_r(\F_\fp)$ if 
\begin{equation}\label{eqSFdim}
\dim_{\F_q}\cM=r\cdot \deg(\fp). 
\end{equation}
 
 \subsection{Anderson motive of Drinfeld module}\label{secMotive} Let $\phi\colon A\to 
 \twist{K}$ be a Drinfeld module over $K$ of rank $r$ defined by 
 $$
 \phi_T=g_0+g_1\tau+\cdots+g_r\tau^r. 
 $$
 We associate with $\phi$ a left $K[T, \tau]$-module $M_\phi$, called the \textit{Anderson motive of $\phi$} \cite{Anderson}, by taking $M_\phi=\twist{K}$ and defining 
 \begin{align*}
 	u\circ m & = um \quad \text{for all $u\in \twist{K}$ and $m\in M_\phi$}, \\ 
 	a\circ m & = m\phi_a \quad \text{for all $a\in A$ and $m\in M_\phi$}, 
 \end{align*}
 where $um$ and $m\phi_a$ are multiplications in $\twist{K}$. By the proof of \cite[Lem. 3.4.4]{PapikianGTM}, $M_\phi$ 
 is freely generated over $K[T]$ by the elements  $\{1, \tau, \dots, \tau^{r-1}\}$; we call this set the \textit{standard basis} of $M_\phi$. 
 
 Let $u\in \End(\phi)$. Then $u$ induces an endomorphism of the $K[T, \tau]$-module $M_\phi$ by $m\mapsto mu$; see the proof of \cite[Prop. 3.4.5]{PapikianGTM}. 
  Let $U=(u_{i,j})_{0\leq i, j\leq r-1}\in \Mat_r(K[T])$ be the matrix 
 by which $u$ acts on $M_\phi$ as a free $K[T]$-module with respect to the standard basis, i.e., 
 $$
 \tau^i u = u_{0, i} + u_{1, i}\tau+\cdots +u_{r-1, i}\tau^{r-1}, \qquad 0\leq i\leq r-1. 
 $$
 By Theorem 3.6.6 and Proposition 3.6.7 in \cite{PapikianGTM}, 
 \begin{equation}\label{eqPu}
 P_{\phi, u}(x)=\det(xI_r-U),
 \end{equation}
 i.e., the characteristic polynomial of $u$ acting on $T_\fp(\phi)$ is equal to the characteristic polynomial of the matrix $U$. 
 In particular, $u$ acts invertibly on $\phi[\fp]$ if and only if $\fp$ does not divide $\det(U)$ in $A$. 
 
 For $f=f_0+f_1T+\cdots+f_nT^n\in K[T]$ and $i\geq 0$, denote $$f^{(q^i)} = f_0^{q^i}+f_1^{q^i}T+\cdots+f_n^{q^i}T^n.$$  
 For a matrix $S=(s_{ij})\in \Mat_{m, n}(K[T])$, put $S^{(q^i)}=\left(s_{ij}^{(q^i)}\right)$. 
 Let $$u=u_0+u_1\tau+\cdots+u_{r-1}\tau^{r-1}$$ be the expansion of $u$ in the standard basis. Then 
 \begin{align*}
 	\tau u & = u_0^{(q)}\tau+u_1^{(q)}\tau^2+\cdots+u_{r-1}^{(q)}\tau^r \\ 
 	& = u_0^{(q)}\tau+u_1^{(q)}\tau^2+\cdots+u_{r-1}^{(q)}(g_r^{-1}(T-g_0-g_1\tau-\cdots-g_{r-1}\tau^{r-1})) \\ 
 	& = u_{r-1}^{(q)}(g_r^{-1}(T-g_0)) + (u_0^{(q)}-u_{r-1}^{(q)}(g_r^{-1}g_1)) \tau+(u_1^{(q)}-u_{r-1}^{(q)}(g_r^{-1}g_2)) \tau^2+\cdots \\ &+
 	(u_{r-2}^{(q)}-u_{r-1}^{(q)}(g_r^{-1}g_{r-1})) \tau^{r-1}. 
 \end{align*}
 Put 
 \begin{equation}\label{eqSphi}
 	S_\phi=\begin{bmatrix} 
 		0 & 0 &\cdots & 0 &(T-g_0)/g_r \\ 
 		1 & 0 & \cdots & 0 &-g_1/g_r \\ 
 		\vdots & \vdots& \cdots &\vdots & \vdots \\ 
 		0 & 0 &\cdots & 1 & -g_{r-1}/g_r
 	\end{bmatrix}. 
 \end{equation}
 The same computation applies to any $m\in M_\phi$ in place of $u$: if $[m]\in K[T]^r$ denotes the column vector of coordinates of $m$ 
 with respect to the standard basis, then $[\tau m]=S_\phi[m]^{(q)}$, and hence 
 $[\tau^j m]=S_\phi S_\phi^{(q)}\cdots S_\phi^{(q^{j-1})}[m]^{(q^j)}$ for all $j\geq 1$. In particular, for any $1\leq i\leq r-1$ we have 
 \begin{align*}
 	\begin{bmatrix} u_{0,i}\\ \vdots \\ u_{r-1, i}\end{bmatrix} = 	S_\phi \begin{bmatrix} u_{0, i-1}\\ \vdots\\ u_{r-1, i-1}\end{bmatrix}^{(q)} & = 
 	S_\phi S_\phi^{(q)}  \begin{bmatrix} u_{0, i-2}\\ \vdots\\ u_{r-1, i-2}\end{bmatrix}^{(q^2)} = \cdots \\ &= 
 	S_\phi S_\phi^{(q)} \cdots  S_\phi^{(q^{i-1})} \begin{bmatrix} u_{0}\\ \vdots\\ u_{r-1}\end{bmatrix}^{(q^i)}. 
 \end{align*}

\begin{example}\label{ssExample}
Let $r>s\geq 1$ be integers with $\gcd(r, s)=1$, so that $n\colonequals \lcm(r, s)=rs$, let $k=\F_{q^n}$, and let $t\in \F_{q^s}$. 
Let $\phi\colon A\to \twist{k}$ be the Drinfeld module defined by 
$$\phi_T=t+\tau^r.$$ 
Let  
$$
\cM=\left\{ a+b\tau^s\mid a, b\in \F_{q^r}\right\} \subset \End(\phi). 
$$
Since $s<r$, the expression $u=a+b\tau^s\in \cM$ is the expansion of $u$ with respect to the standard basis of $M_\phi$. In that case,   
one computes that the matrix $U$ by which $u$ acts on $M_\phi$ with respect to the standard basis is  
\begin{equation}\label{eqU2}
	U=\diag(a, a^q, \dots, a^{q^{r-1}})+\begin{bmatrix} & D_{s}\\ I_{r-s} & \end{bmatrix}\diag(b, b^q, \dots, b^{q^{r-1}}), 
\end{equation}
where 
$$
D_s =\diag(T-t, T-t^q, \cdots, T-t^{q^{s-1}})\in \Mat_{s}(k[T]). 
$$
We claim that 
\begin{equation}\label{eqDetU2}
\det U =\Nr_{\F_{q^r}/\F_q}(a) +(-1)^{r-1} \Nr_{\F_{q^r}/\F_q}(b)\cdot \Nr_{\F_{q^s}/\F_q}(T-t).  
\end{equation}
Indeed, indexing the rows and columns of $U$ by $\Z/r\Z$, the nonzero entries of $U$ lie 
on the main diagonal and in the positions $(j+s, j)$, $j\in \Z/r\Z$. A permutation contributing 
to $\det U$ is determined by the set $\cR\subseteq \Z/r\Z$ of columns in which the off-diagonal entry 
is selected, and the selected entries lie in distinct rows if and only if $i\in \cR \Leftrightarrow i-s\in \cR$ for all 
$i\in \Z/r\Z$, i.e., if and only if $\cR$ is a union of orbits of the translation $i\mapsto i+s$ on $\Z/r\Z$. 
These orbits are the cosets of the subgroup generated by $s$, so the assumption $\gcd(r,s)=1$ implies 
that the only possibilities are $\cR=\emptyset$ and $\cR=\Z/r\Z$, which give the two terms in \eqref{eqDetU2}. 
(When $g'\colonequals \gcd(r, s)>1$, there are $2^{g'}$ contributing permutations, one for each union of the $g'$ orbits, 
and $\det U$ acquires cross terms mixing $a$ and $b$; for example, for $r=4$, $s=2$, one computes 
$\det U = \Nr(a)+\Nr(b)(T-t)(T-t^q)-a^{q+q^3}b^{1+q^2}(T-t)-a^{1+q^2}b^{q+q^3}(T-t^q)$.) 

We want to show that for an appropriate choice of the parameters, all $0\neq u\in \cM$ act invertibly on $\phi[\fp]$. 
Fix a prime $\fp\lhd A$ with $\deg(\fp)=2$ and fix a class $c\in \F_\fp\setminus \F_q$. 
By Dirichlet's theorem for $A$, there is $m_0$ such that for every $m\geq m_0$ there exists a monic prime 
$\fq\equiv c \pmod{\fp}$ of degree $m$. Now choose $s\geq m_0$ and $r>s$ with $\gcd(r, s)=1$, let $\fq$ be such a prime of degree $s$, 
and let $t\in \F_{q^s}$ be a root of $\fq$, so that $\fq=\Nr_{\F_{q^s}/\F_q}(T-t)$.  

If $u\in \cM$ is not invertible on $\phi[\fp]$, then $\fp$ divides
$$\det(U)= \Nr_{\F_{q^r}/\F_q}(a) +(-1)^{r-1} \Nr_{\F_{q^r}/\F_q}(b)\cdot \fq.$$
 If $b=0$, this is obviously not possible, so we assume $b\neq 0$. With this assumption, if $\fp$ 
divides $\det(U)$, then $\fq$ is congruent modulo $\fp$ to an element in $\F_q$, which contradicts $\fq\equiv c\pmod{\fp}$. 
Thus, the images of nonzero elements of $\cM$ in $\End_{\F_\fp}(\phi[\fp])$ are invertible. Finally, note that $\dim_{\F_q}(\cM)=2r$, so 
by \eqref{eqSFdim} we get a semifield code, since $\deg(\fp)=2$. 
\end{example}

\subsection{Nuclear parameters}\label{ssNP} In this subsection we recall the nuclear parameters of a rank-metric code and 
compare them, for codes in $\End_{\F_\fp}(\phi[\fp])$, computed over $\F_q$ and over $\F_\fp$. 
Throughout, $\phi$ is a Drinfeld module of rank $r$ over $K$, 
$\fp\neq \chr_A(\phi)$ is a prime of degree $d$, and $\phi[\fp]\cong \F_\fp^r\cong \F_q^{rd}$. 

\begin{defn}\label{defNPq} Let $\cM\subseteq \End_{\F_q}(\phi[\fp])$ be an $\F_q$-subspace. Following \cite[Def. 2]{Sheekey}, define
	the \textit{left (right) idealizer, centralizer, and center} of $\cM$ by
	\begin{align*}
		I_l(\cM) &=\{f\in \End_{\F_q}(\phi[\fp])\mid f\cM\subseteq \cM\},\\
		I_r(\cM) &=\{f\in \End_{\F_q}(\phi[\fp])\mid \cM f\subseteq \cM\},\\
		C(\cM) &=\{f\in \End_{\F_q}(\phi[\fp])\mid mf=fm \text{ for all }m\in \cM\},\\
		Z(\cM) &=I_l(\cM)\cap C(\cM).
	\end{align*}
	If $\cM$ is a semifield code containing the identity, then each of these sets is a finite field containing $\F_q$ \cite[p. 436]{Sheekey}, and the
	\textit{nuclear parameters} of $\cM$ are
	\[
	(\dim_{\F_q}\cM;\ \dim_{\F_q} I_l(\cM), \dim_{\F_q} I_r(\cM), \dim_{\F_q} C(\cM);\ \dim_{\F_q} Z(\cM)).
	\]
	If $\cS$ is the semifield corresponding to $\cM$, then $I_l(\cM)$, $I_r(\cM)$, $C(\cM)$, $Z(\cM)$ are isomorphic to the
	left nucleus, the middle nucleus, the right nucleus, and the center of $\cS$, respectively; see \cite[Prop. 5]{Sheekey} and \cite{MarPol}. 
	By \cite[Prop. 4]{Sheekey}, equivalent semifield codes containing the identity have the same nuclear parameters (for codes which do
	not contain the identity this can fail for the centralizer and the center). For an arbitrary semifield code $\cM$, we define 
	the nuclear parameters of $\cM$ to be those of $\cM X^{-1}$, where $X$ is a nonzero element of $\cM$. The codes $\cM X^{-1}$ contain the 
	identity and are equivalent to $\cM$, so this does not depend on the choice of $X$, and with this definition equivalent semifield codes 
	have the same nuclear parameters.
\end{defn}

The torsion module $\phi[\fp]$ is most naturally an $\F_\fp$-vector space, as this reflects the $A$-module structure given by $\phi$, 
and the codes constructed in this paper lie in $\End_{\F_\fp}(\phi[\fp])\subseteq \End_{\F_q}(\phi[\fp])$. It is therefore 
convenient to also consider, for an $\F_q$-subspace $\cM\subseteq \End_{\F_\fp}(\phi[\fp])$, the sets 
\begin{align*}
I_l^\fp(\cM)&=I_l(\cM)\cap \End_{\F_\fp}(\phi[\fp]), & I_r^\fp(\cM)&=I_r(\cM)\cap \End_{\F_\fp}(\phi[\fp]), \\
C^\fp(\cM)&=C(\cM)\cap \End_{\F_\fp}(\phi[\fp]), & Z^\fp(\cM)&=I_l^\fp(\cM)\cap C^\fp(\cM), 
\end{align*}
i.e., the idealizers, centralizer and center of $\cM$ computed inside $\End_{\F_\fp}(\phi[\fp])$. The next proposition 
compares the two versions. 

\begin{prop}\label{lemCompare} Let $\cM\subseteq \End_{\F_\fp}(\phi[\fp])$ be an $\F_q$-subspace with $1\in \cM$. Then 
	$$
	I_l(\cM)=I_l^\fp(\cM)\subseteq \cM, \qquad I_r(\cM)=I_r^\fp(\cM)\subseteq \cM, \qquad Z(\cM)=Z^\fp(\cM). 
	$$
	Moreover, if $\cM$ is a semifield code, then $C(\cM)=C^\fp(\cM)$. 
\end{prop}
\begin{proof}
	If $f\in I_l(\cM)$, then $f=f\cdot 1\in \cM\subseteq \End_{\F_\fp}(\phi[\fp])$, which proves the claims about $I_l(\cM)$; 
	the same argument applies to $I_r(\cM)$. Hence 
	\[Z(\cM)=I_l(\cM)\cap C(\cM)=I_l^\fp(\cM)\cap C(\cM)=I_l^\fp(\cM)\cap C^\fp(\cM)=Z^\fp(\cM).\] 
	The elements of $\cM$ are $\F_\fp$-linear, so they commute with the scalars $\F_\fp\cdot 1$, where $\F_\fp$ acts on $\phi[\fp]$ through $\phi$; 
	thus $\F_\fp\cdot 1\subseteq C(\cM)$. If $\cM$ is a semifield code, then $C(\cM)$ is a field, so every $f\in C(\cM)$ commutes 
	with $\F_\fp\cdot 1$, i.e., $f$ is $\F_\fp$-linear. Hence $C(\cM)=C^\fp(\cM)$. 
\end{proof}

Thus the nuclear parameters of a semifield code $\cM\subseteq \End_{\F_\fp}(\phi[\fp])$ containing the identity are the same whether 
they are computed over $\F_q$ or over $\F_\fp$. Without the semifield assumption, $C(\cM)$ can be larger than $C^\fp(\cM)$; e.g., 
for $\cM=\F_q\cdot 1$ and $d>1$ we have $C(\cM)=\End_{\F_q}(\phi[\fp])\neq \End_{\F_\fp}(\phi[\fp])=C^\fp(\cM)$.


\section{A duality between a Drinfeld module and its endomorphisms}\label{secDuality}
Let $\phi\colon A\to \twist{K}$ be a Drinfeld module of rank $r$ over a field $K\supseteq \F_q$, and let $u\in \End(\phi)$
with $s\colonequals \deg_\tau(u)\geq 1$. Then $\psi\colon \F_q[x]\to \twist{K}$, $\psi_x=u$, is a Drinfeld module of rank $s$
(for the polynomial ring $\F_q[x]$ in place of $A=\F_q[T]$), and $\phi_T\in \End(\psi)$, since $\phi_T$ and $u$ commute.
Everything said in $\S$\ref{secTM} and $\S$\ref{secMotive} applies to $\psi$, with $x$ in place of $T$. Thus, the roles
of $\phi_T$ and $u$ are symmetric, and in this section we make this symmetry precise. It is natural to think of the pair $(\phi,\psi)$ as two \textit{entangled} Drinfeld modules. 
Most of the results of this section are known, at least when $K$ is finite and $u$ is the Frobenius endomorphism; 
see the remarks below for precise references. Our proofs are different: they use the Anderson motive and the symmetry 
between $\phi$ and $\psi$, and they work over an arbitrary field $K$. The results will be used
in Sections \ref{secSheekey}, \ref{secHigher} and \ref{s2C}.

Since $\phi_T$ and $u$ commute, $B\colonequals \F_q[\phi_T, u]\subseteq \twist{K}$ is a commutative integral domain. Let $E=\F_q(\phi_T, u)$
be its field of fractions, and let
$$
F\colonequals \F_q(\phi_T), \qquad F'\colonequals \F_q(u).
$$
We identify $F$ with $\F_q(T)$ via $\phi$ and $F'$ with $\F_q(x)$ via $\psi$. Since $\End(\phi)$ is a finitely generated $A$-module,
$u$ is integral over $A=\F_q[\phi_T]$, so the minimal polynomial $m_1(x)$ of $u$ over $F$ is a monic polynomial in $x$
with coefficients in $A$; similarly, the minimal polynomial $m_2(T)$ of $\phi_T$ over $F'$ is a monic polynomial in $T$ with
coefficients in $\F_q[u]=\F_q[x]$. Both $m_1(x)$ and $m_2(T)$ are elements of $\F_q[T, x]$.
Note that we use $x$ both as the variable of $\psi$ and as the variable of the minimal polynomial of $u$. This is consistent 
with the identification $\F_q[u]=\F_q[x]$ via $\psi$: throughout, $T$ and $x$ stand for $\phi_T$ and $u$, respectively, 
and a polynomial $f(T, x)\in \F_q[T, x]$ vanishes on the pair $(\phi_T, u)$ if and only if $f(\phi_T, u)=0$ in $\twist{K}$.

\begin{lem}\label{lemMin} There is $c\in \F_q^\times$ such that $m_2(T)=c\cdot m_1(x)$ in $\F_q[T, x]$. In particular,
	$B\cong \F_q[T, x]/(m_1)$, $\deg_x m_1=[E:F]$, and $\deg_T m_1=[E:F']$.
\end{lem}
\begin{proof}
	Let $I$ be the kernel of the homomorphism $\F_q[T, x]\to \twist{K}$, $T\mapsto \phi_T$, $x\mapsto u$, so that $B\cong \F_q[T, x]/I$.
	Clearly $m_1\in I$. Conversely, let $g\in I$. Since $m_1$ is monic in $x$ with coefficients in $A$, division with remainder
	in $A[x]$ gives $g=m_1 h+h'$ with $h, h'\in A[x]$ and $\deg_x h'<\deg_x m_1$. Then $h'(\phi_T, u)=0$, i.e.,
	$h'(x)\in F[x]$ annihilates $u$, so $h'=0$ by the minimality of $m_1$. Hence $I=m_1\F_q[T, x]$. By symmetry,
	$I=m_2\F_q[T, x]$. Two generators of the same principal ideal of $\F_q[T, x]$ differ by a unit, i.e., by an element of $\F_q^\times$.
\end{proof}

\begin{rem} When $K$ is finite and $u=\pi$, Lemma \ref{lemMin} is \cite[Lem. 4.3.1]{PapikianGTM}:
	$m_1(x)$, viewed as a polynomial in $T$ with coefficients in $\F_q[\pi]$, is irreducible over
	$\F_q(\pi)$ of degree $[E:F']$, with leading coefficient in $\F_q^\times$. The proof in
	\cite{PapikianGTM} relies on the analogue of the Riemann hypothesis for Drinfeld modules
	over finite fields. The argument above is purely
	algebraic and applies to any commuting pair $\phi_T, u$ over an arbitrary field $K$.
\end{rem}

\begin{prop}\label{lemCharMin} $P_{\phi, u}(x)=m_1(x)^{r/[E:F]}$ in $A[x]$.
\end{prop}
\begin{proof} Let $U\in \Mat_r(K[T])$ be the matrix by which $u$ acts on the motive $M_\phi$, as in $\S$\ref{secMotive}, so that
	$P_{\phi, u}(x)=\det(xI_r-U)$ by \eqref{eqPu}. In $M_\phi$, the variable $T$ acts by $m\mapsto m\phi_T$ and $u$ acts by $m\mapsto mu$;
	these two $K[T]$-linear maps commute, and the first one is multiplication by the scalar $T$. Hence, for any $f(T, x)\in \F_q[T, x]$,
	the matrix $f(T\cdot I_r, U)$ represents the map $m\mapsto m f(\phi_T, u)$. Taking $f=m_1$, we get $m_1(U)=0$. Therefore,
	every eigenvalue of $U$ in an algebraic closure $\Omega$ of $K(T)$ is a root of $m_1(x)$.

	Now let $g(x)\in F[x]$ be a monic irreducible factor of $P_{\phi, u}(x)$. A root of $g$ in $\Omega$ is an eigenvalue of $U$,
	hence a root of $m_1$. Thus $g$ and $m_1$ have a nontrivial common factor in $\Omega[x]$, and therefore also in $F[x]$,
	because the greatest common divisor of two polynomials does not change under field extensions. Since $g$ and $m_1$ are
	monic and irreducible in $F[x]$, we conclude that $g=m_1$. Hence $P_{\phi, u}(x)=m_1(x)^e$ for some $e\geq 1$, and comparing
	degrees gives $e\cdot [E:F]=r$.
\end{proof}

\begin{rem} When $K$ is finite and $u=\pi$, Proposition \ref{lemCharMin} is due to Gekeler \cite[$\S$3]{GekelerFDM}; 
	see also \cite[Thm. 4.2.2]{PapikianGTM}. For an arbitrary endomorphism $u$ over a finite field it is \cite[Thm. 4.2.5]{PapikianGTM}, 
	whose proof adapts Gekeler's argument, using the Tate module $V_\fl(\phi)$ and the reduced norm of the division algebra 
	$\End(\phi)\otimes_A F$. The argument above, which uses the Anderson motive instead, is shorter and works over an arbitrary field $K$. 
\end{rem}

\begin{lem}\label{lemDegDet}
	Let $0\neq u\in \End(\phi)$, and let $U$ be the matrix of $u$ acting on $M_\phi$. Then $\deg_T\det(U)=\deg_\tau(u)$.
\end{lem}
\begin{proof}
	This is \cite[(3.6.12)]{PapikianGTM}. We give a direct argument. The map $m\mapsto mu$ on $M_\phi$ is injective, since $\twist{K}$
	has no zero divisors, so $\det(U)\neq 0$. Its image is $\twist{K}u$, and by the right division algorithm in $\twist{K}$ the cokernel
	$\twist{K}/\twist{K}u$ has dimension $\deg_\tau(u)$ over $K$. On the other hand, by the Smith normal form over the principal ideal domain
	$K[T]$, the cokernel of a $K[T]$-linear endomorphism of $K[T]^r$ with nonzero determinant is isomorphic to $\bigoplus_{i=1}^rK[T]/(\delta_i)$
	with $\delta_1\cdots\delta_r\in K^\times\cdot\det(U)$, so its dimension over $K$ is $\deg_T\det(U)$.
\end{proof}

\begin{prop}\label{propDeg}
	Let $u, w\in \twist{K}$ be commuting elements of positive $\tau$-degree, and let $E=\F_q(u, w)$. Then
	$$
	\frac{\deg_\tau(u)}{\deg_\tau(w)} = \frac{[E:\F_q(u)]}{[E:\F_q(w)]}.
	$$
\end{prop}
\begin{proof} Let $r=\deg_\tau(u)$ and $r'=\deg_\tau(w)$. Let $\varphi$ be the Drinfeld module of rank $r$ over $K$ defined by $\varphi_T=u$.
Then $w\in \End(\varphi)$. Let $W$ be the matrix of $w$ acting on $M_\varphi$, and let $m(T, x)\in \F_q[T, x]$ be the minimal polynomial of $w$
over $\F_q(u)$, where $T$ and $x$ stand for $u$ and $w$. By Lemma \ref{lemDegDet}, $\deg_T\det(W)=r'$. On the other hand, by Proposition
\ref{lemCharMin} applied to $(\varphi, w)$,
$$
\det(W)=(-1)^rP_{\varphi, w}(0)=(-1)^r m(T, 0)^{r/[E:\F_q(u)]}.
$$
By Lemma \ref{lemMin} applied to $(\varphi, w)$, $m$ is an $\F_q^\times$-multiple of the minimal polynomial of $u$ over $\F_q(w)$, which is monic in $T$
of degree $[E:\F_q(w)]$ with coefficients in $\F_q[x]$. Hence $\deg_Tm(T, 0)=[E:\F_q(w)]$, and therefore
\[
r'=\frac{r\cdot [E:\F_q(w)]}{[E:\F_q(u)]}. \qedhere
\]
\end{proof}

We return to the Drinfeld module $\phi$ of rank $r$ and its endomorphism $u$ of degree $s$, with $\psi_x=u$. Let
$$
P_1(x)\colonequals P_{\phi, u}(x)\in A[x]\qquad \text{and} \qquad P_2(T)\colonequals P_{\psi, \phi_T}(T)\in \F_q[x][T]
$$
be the characteristic polynomials of $u\in \End(\phi)$ and of $\phi_T\in \End(\psi)$, respectively (cf. $\S$\ref{secTM}).
Both are elements of $\F_q[T, x]$: $P_1$ is monic of degree $r$ in $x$, and $P_2$ is monic of degree $s$ in $T$.

\begin{thm}\label{thmDuality} With the above notation:
	\begin{enumerate}
		\item $e\colonequals r/[E:F]=s/[E:F']$ is a positive integer.
		\item $P_1(x)=m_1(x)^e$ and $P_2(T)=m_2(T)^e$. Consequently, $P_1$ and $P_2$ coincide up to a factor in $\F_q^\times$.
	\end{enumerate}
\end{thm}
\begin{proof}
	Proposition \ref{lemCharMin} applied to $(\phi, u)$ gives $P_1=m_1^{r/[E:F]}$, and applied to $(\psi, \phi_T)$ gives
	$P_2=m_2^{s/[E:F']}$. Proposition \ref{propDeg}, applied to $\phi_T$ and $u$, gives $r/s=[E:F]/[E:F']$, so the two exponents
	are equal; this proves (1) and the first claim in (2). The second claim in (2) follows from Lemma \ref{lemMin}.
\end{proof}

Let $\overline{m}_2(T)\in \F_q[T]$ and $\overline{P}_2(T)\in \F_q[T]$ be the polynomials obtained from $m_2(T)$ and $P_2(T)$ by reducing
their coefficients modulo $x$, i.e., by setting $x=0$ in the corresponding elements of $\F_q[T, x]$. Let $U$ be the matrix of $u$ acting on $M_\phi$,
so that $\det(U)=(-1)^rP_1(0)$. By Theorem \ref{thmDuality},
\begin{equation}\label{eqPbar}
\overline{P}_2(T)=\overline{m}_2(T)^e \qquad \text{and}\qquad \det(U)=c\cdot \overline{P}_2(T) \quad \text{for some } c\in \F_q^\times.
\end{equation}
\begin{lem}\label{lemKer}
Assume $\partial(u)\neq 0$. Then $\ker(u)=\psi[x]$ is an $\F_q$-vector space of dimension $s$, stable under $\phi_T$, and 
the characteristic polynomial of $\phi_T$ acting on $\ker(u)$ is $\overline{P}_2(T)$. In particular, $\det(U)$ is, up to a 
factor in $\F_q^\times$, the characteristic polynomial of $\phi_T$ acting on $\ker(u)$.
\end{lem}
\begin{proof}
Since $\partial(u)\neq 0$, we have $x\neq \chr_{\F_q[x]}(\psi)$. Hence $\ker(u)=\ker(\psi_x)$ is the $x$-torsion $\psi[x]$ of $\psi$, 
which is a free module of rank $s$ over $\F_q[x]/(x)=\F_q$. It is stable under $\phi_T$, since $\phi_T$ commutes with $u$. 
Applying $\S$\ref{secTM} to $\psi$, its endomorphism $\phi_T$ and the prime $x$, the characteristic polynomial of $\phi_T$ acting 
on $\psi[x]$ is the reduction modulo $x$ of $P_2(T)=P_{\psi, \phi_T}(T)$, i.e., $\overline{P}_2(T)$. The last claim follows 
from \eqref{eqPbar}.
\end{proof}

\begin{cor}\label{corCLM}
	$\phi_{\overline{m}_2}=w u$ for some $w\in B=\F_q[\phi_T, u]\subseteq \End(\phi)$. Consequently, $\phi_{\overline{P}_2}$ and
	$\phi_{\det(U)}$ are also left (and right) multiples of $u$ in $\End(\phi)$.
\end{cor}
\begin{proof}
	We have $m_2(T)-\overline{m}_2(T)=x\cdot Q(T, x)$ for some $Q\in \F_q[T, x]$. Evaluating at $(T, x)=(\phi_T, u)$ and using
	$m_2(\phi_T, u)=0$, we get $\overline{m}_2(\phi_T)=-Q(\phi_T, u)u$. The other claims follow from \eqref{eqPbar}, since $B$ is commutative.
\end{proof}

\begin{cor}\label{corInv}
	Let $\fp\lhd A$ be a prime different from $\chr_A(\phi)$. The following are equivalent:
	\begin{enumerate}
		\item[(i)] $u$ does not act invertibly on $\phi[\fp]$;
		\item[(ii)] $\fp$ divides $\det(U)$ in $A$;
		\item[(iii)] $\fp$ divides $\overline{P}_2(T)$;
		\item[(iv)] $\fp$ divides $\overline{m}_2(T)$.
	\end{enumerate}
	If $\partial(u)\neq 0$, these conditions are also equivalent to:
	\begin{enumerate}
	\item[(v)] $\fp$ divides the characteristic polynomial of $\phi_T$ acting on the $\F_q$-vector space $\ker(u)$.
	\end{enumerate}
\end{cor}
\begin{proof}
	(i)$\Leftrightarrow$(ii) was observed in $\S$\ref{secMotive}; (ii)$\Leftrightarrow$(iii)$\Leftrightarrow$(iv) follow from
	\eqref{eqPbar}, since $\fp$ is prime; (iii)$\Leftrightarrow$(v) is Lemma \ref{lemKer}.
\end{proof}

We now specialize to the case where the base field is finite. Let $k=\F_{q^n}$, let $\pi\colonequals\tau^n$, 
which lies in the center of $\twist{k}$, and put 
$$
N(a)\colonequals \Nr_{k/\F_q}(a)=a^{1+q+\cdots+q^{n-1}},\qquad a\in k.
$$
Let $u=u_0+u_1\tau+\cdots+u_s\tau^s\in \twist{k}$ with $s\geq 1$ and $u_s\neq 0$, and let $\psi_x=u$. Then $\pi\in \End(\psi)$. 
Equivalently, $u$ is an endomorphism of the Drinfeld module $\phi$ of rank $n$ defined by $\phi_T=\pi$, and for this pair 
$P_2(T)=P_{\psi, \pi}(T)$ is the characteristic polynomial of the Frobenius endomorphism of $\psi$. 

\begin{prop}\label{propNorm} Assume $u_0\neq 0$. Then $\ker(u)$ is an $s$-dimensional $\F_q$-vector space stable under $\pi$, and 
\begin{equation}\label{eqNorm}
N(u_0/u_s)=(-1)^{sn}\det(\pi\mid \ker u)=(-1)^{s(n-1)}\,\overline{P}_2(0). 
\end{equation}
\end{prop}
\begin{proof}
Since $\partial(u)=u_0\neq 0$, Lemma \ref{lemKer} shows that $\ker(u)=\psi[x]$ has dimension $s$ over $\F_q$, 
it is stable under $\phi_T=\pi$, and $\overline{P}_2(T)$ is the characteristic polynomial of $\pi$ acting on it. In particular, 
$\det(\pi\mid \ker u)=(-1)^s\overline{P}_2(0)$, which is the second equality in \eqref{eqNorm}. 

Let $S_\psi\in \Mat_s(k[x])$ be the matrix \eqref{eqSphi} attached to $\psi$, i.e., with $x, u_i, s$ in place of $T, g_i, r$. 
The calculations in $\S$\ref{secMotive}, applied to $\psi$, show that $\pi$ acts on $M_\psi$ by the matrix 
$$
S_{u, \pi}\colonequals S_\psi S_\psi^{(q)}\cdots S_\psi^{(q^{n-1})}, 
$$
so $P_2(T)=\det(TI_s-S_{u, \pi})$ by \eqref{eqPu}. Denote by $\overline{S}_\psi$ and $\overline{S}_{u, \pi}$ the reductions of 
these matrices modulo $x$. Then $\overline{P}_2(T)=\det(TI_s-\overline{S}_{u, \pi})$ and 
$\overline{S}_{u, \pi}=\overline{S}_\psi\overline{S}_\psi^{(q)}\cdots \overline{S}_\psi^{(q^{n-1})}$. Since 
$\det \overline{S}_\psi=(-1)^{s+1}(-u_0/u_s)=(-1)^s u_0/u_s$, we get 
$\det \overline{S}_{u, \pi}=(-1)^{sn}N(u_0/u_s)$, and therefore 
$\overline{P}_2(0)=(-1)^s\det \overline{S}_{u, \pi}=(-1)^{s(n+1)}N(u_0/u_s)$. 
\end{proof}

\begin{rem}\label{remNorm} 
Proposition \ref{propNorm} can also be deduced from the formula for the constant term of the characteristic 
polynomial of Frobenius. The $\F_q[x]$-characteristic of $\psi$ is the minimal polynomial $\fl(x)\in \F_q[x]$ of $u_0$ over $\F_q$. 
Let $d'=\deg(\fl)$. By \cite[Thm. 4.2.7(3)]{PapikianGTM} applied to $\psi$, 
$$
P_2(0)=(-1)^{sn-s-n}N(u_s)^{-1}\fl(x)^{n/d'}\quad \text{in } \F_q[x]. 
$$
Since $u_0\neq 0$, we have $\fl\neq x$. Setting $x=0$ and using 
$N(u_0)=\Nr_{\F_q(u_0)/\F_q}(u_0)^{n/d'}=(-1)^n\fl(0)^{n/d'}$, we obtain \eqref{eqNorm}. 
Thus \eqref{eqNorm} is the reduction modulo $x$ of the formula for the constant term of the characteristic polynomial 
of the Frobenius of $\psi$.
\end{rem}

\medskip

We end this section with another application of the entanglement of a Drinfeld module with its Frobenius used
before Proposition \ref{propNorm}, now at the characteristic: we relate the reduction of the characteristic
polynomial of the Frobenius modulo $\chr_A(\phi)$ to the height of $\phi$.
Let $\phi$ be a Drinfeld module of rank $r$ over $k=\F_{q^n}$, let $\fp=\chr_A(\phi)$, $d=\deg(\fp)$, and $t=\gamma(T)$,
so that $\fp(t)=0$ and $A/\fp\cong \F_q(t)\subseteq k$. Let $P(x)=P_{\phi, \pi}(x)$ be the characteristic polynomial of the
Frobenius endomorphism $\pi=\tau^n$ of $\phi$, let $\overline{P}(x)\in \F_\fp[x]$ be its reduction modulo $\fp$, and let 
$H=H(\phi)$ be the height of $\phi$.

\begin{prop}\label{propHeight}
$\overline{P}(x)=x^{H}\cdot g(x)$ for some $g(x)\in \F_\fp[x]$ with $g(0)\neq 0$.
\end{prop}

\begin{proof}
Let $\psi\colon \F_q[x]\to \twist{k}$ be the entangled Drinfeld module $\psi_x=\pi$. The motive $M_\psi=\twist{k}$ is a free $k[x]$-module with basis
$1, \tau, \dots, \tau^{n-1}$, where $k$ acts by left multiplication and $x$ by right multiplication by $\pi$. For $w\in \twist{k}$,
let $R_w\colon m\mapsto mw$, which is $k[x]$-linear, and for $a\in k$ let $L_a\colon m\mapsto am$. Left and right multiplications commute.
By Theorem \ref{thmDuality}, applied to $\phi$ and $u=\pi$, we have $P(x)=c\cdot P_2(T, x)$ with $c\in \F_q^\times$, where
$P_2(T)=P_{\psi, \phi_T}(T)$. Reducing modulo $\fp$ amounts to setting $T=t$, so $\overline{P}(x)=c\cdot P_2(t, x)$, and by \eqref{eqPu} 
applied to $\psi$ and $\phi_T\in \End(\psi)$ we have $(-1)^nP_2(t, x)=\det_{k[x]}\left(R_{\phi_T}-L_t\mid M_\psi\right)$.

We claim that
\begin{equation*}
\ord_x\det_{k[x]}(R_w)=\ord_\tau(w)\qquad \text{for every } 0\neq w\in \twist{k},
\end{equation*}
where $\ord_\tau(w)\colonequals\Ht(w)$.
Write $w=v\tau^{\ord_\tau(w)}$ with $\partial(v)\neq 0$, so that $R_w=R_{\tau^{\ord_\tau(w)}}R_v$. Since $\det(R_\tau)^n=\det(R_{\tau^n})=\det(xI_n)=x^n$,
we have $\ord_x\det(R_\tau)=1$. Modulo $x$, the operator $R_v$ acts on $\twist{k}/\twist{k}\tau^n$ by
$\tau^iv=v_0^{q^i}\tau^i+(\text{terms of higher degree})$, so it is triangular with diagonal entries $v_0^{q^i}\neq 0$, and
$\det(R_v)\equiv N(v_0)\not\equiv 0\pmod x$. This proves the claim. 

The operators $R_{\phi_T}$ and $L_{t^{q^j}}$, $0\leq j\leq d-1$, commute, and $\prod_j(X-t^{q^j})=\fp(X)$. Hence
$$
\prod_{j=0}^{d-1}\left(R_{\phi_T}-L_{t^{q^j}}\right)=\fp(R_{\phi_T})=R_{\phi_\fp}.
$$
By the formula for $(-1)^nP_2(t, x)$ above, the determinant of the $j$-th factor is $(-1)^nP_2(t^{q^j}, x)$, which is obtained from $(-1)^nP_2(t, x)$ by
raising the coefficients to the $q^j$-th power, since $P_2(T, x)\in \F_q[T, x]$. Thus all $d$ factors have the same order at $x$,
and the claim of the previous paragraph gives $d\cdot \ord_xP_2(t, x)=\ord_\tau(\phi_\fp)=dH$. 
Thus $\ord_x\overline{P}(x)=H$, i.e., $\overline{P}(x)=x^Hg(x)$ with $g(0)\neq 0$.
\end{proof}

\begin{rem} Proposition \ref{propHeight} is \cite[Thm. 4.2.13]{PapikianGTM},
whose proof is local: it uses Newton polygons and the places of $F(\pi)$ over $\fp$.
\end{rem}


\section{Reinterpretation of Sheekey's construction}\label{secSheekey} Let $n\geq 1$ be an integer, let $k=\F_{q^n}$ be the degree
$n$ extension of $\F_q$, and let $R=\twist{k}$ be the twisted polynomial ring with coefficients in $k$.
The  center of $R$ is $A'\colonequals \F_q[\pi]$, where $\pi\colonequals \tau^n$. Note that $R$ is a free $A'$-algebra of rank $n^2$.
Let $\phi$ be the Drinfeld module of rank $n$ defined by $\phi_T=\pi$. Note that $\chr_A(\phi)=T$ and for any $a(T)\in A$, we have $\phi_a=a(\pi)$.
For this Drinfeld module the endomorphism ring is as large as possible, $\End(\phi)=R$, since $\pi$
is in the center of $R$.

Let $u\in R$ with $s\colonequals \deg_\tau(u)\geq 1$, and let $\psi_x=u$ be the entangled Drinfeld module, as in Section \ref{secDuality}.
We use the notation of that section for the pair $(\phi, u)$: $m_2(T)$ is the minimal polynomial of $\pi=\phi_T$ over $\F_q(u)=\F_q(x)$,
$P_2(T)$ is the characteristic polynomial of the Frobenius endomorphism $\pi\in \End(\psi)$, and $\overline{m}_2(T), \overline{P}_2(T)\in \F_q[T]$
are their reductions modulo $x$. Put $d_u\colonequals[\F_q(u, \pi):\F_q(u)]=\deg_T m_2(T)$, so that $P_2(T)=m_2(T)^{s/d_u}$ 
by Theorem \ref{thmDuality}.

\begin{defn} A \textit{central left multiple} of nonzero $u\in R$ is a polynomial $f\in A'$ such that $f=w u$ for some $w\in R$.
	The \textit{minimal central left multiple} of $u$ is the unique central left multiple of $u$ which is monic and has minimal degree in $\pi$.
\end{defn}

\begin{rem} By Corollary \ref{corCLM}, $\overline{m}_2(\pi)$ is a central left multiple of $u$. The difference of two central left multiples
	is again a central left multiple, so the minimal central left multiple is unique, and, by the division algorithm in $A'$,
	it divides every central left multiple of $u$.
\end{rem}

\begin{lem}\label{prop1.8} Assume $\partial(u)\neq 0$.
	The polynomial $u$ is irreducible in $R$ if and only if $\overline{P}_2(T)$ is irreducible in $\F_q[T]$.
\end{lem}
\begin{proof}
	Since $\partial(u)\neq 0$, the polynomial $\overline{P}_2(T)$ is the characteristic polynomial of $\pi$ acting on
	$\psi[x]=\ker(u)$ by Lemma \ref{lemKer}.
	Therefore, $\overline{P}_2(T)$ is reducible in $\F_q[T]$ if and only if $\ker(u)$ has a nonzero proper $\pi$-invariant $\F_q$-subspace.
	An $\F_q$-subspace $W\subseteq \ker(u)$ gives the $\F_q$-linear polynomial $w_W=\prod_{\alpha\in W}(z-\alpha)\in \overline{k}[z]$, which we
	regard as an element of $\twist{\overline{k}}$. By \cite[Lem. 3.1.16]{PapikianGTM}, there is $g\in \twist{\overline{k}}$ such that $u=g w_W$.
	The subspace $W$ is $\pi$-invariant if and only if the coefficients of $w_W$ are in $k$, and in that case $g\in R$ by the division algorithm in $R$.
	Conversely, if $u=gw$ with $g, w\in R$ of positive degree, then $w$ is separable, since $\partial(g)\partial(w)=\partial(u)\neq 0$,
	so $\ker(w)$ is a nonzero proper $\pi$-invariant subspace of $\ker(u)$.
	Thus, $\overline{P}_2(T)$ is reducible in $\F_q[T]$ if and only if $u$ is reducible in $R$.
\end{proof}

\begin{thm}\label{thm2.5} Assume $\partial(u)\neq 0$.
	If $u$ is irreducible in $R$, then the minimal central left multiple of $u$ is $\overline{P}_2(\pi)$.
\end{thm}
\begin{proof}
	By Lemma \ref{prop1.8}, $\overline{P}_2(T)$ is irreducible. Since $\overline{P}_2(T)=\overline{m}_2(T)^{s/d_u}$, it follows that
	$\overline{m}_2(T)=\overline{P}_2(T)$. The minimal central left multiple of $u$ is nonconstant, since $\deg_\tau(u)\geq 1$,
	and it divides the central left multiple $\overline{m}_2(\pi)$. As $\overline{m}_2$ is irreducible, the two are equal.
\end{proof}

\begin{rem} Theorem \ref{thm2.5} corresponds to the case of Theorem 3 in \cite{Sheekey} where the element is irreducible. We note that the assumption in \cite{Sheekey} that
$u$ is monic is unnecessary; on the other hand, the second assertion of that theorem
(that the minimal polynomial of the matrix $A_f$ equals its characteristic polynomial when $K$ is finite)
requires $f$ to be irreducible, since its proof invokes Theorem 2(iii) of \textit{loc. cit.}, where irreducibility is assumed. 
For example, if $n>1$ and $u=\tau^n-1$, then the minimal central left multiple of $u$ is $\pi-1$, whereas $\overline{P}_2(T)=(T-1)^n$, 
because $\pi$ acts as the identity on $\ker(u)=k$.
\end{rem}

Let $\fp\in  A$ be a monic irreducible polynomial of degree $d$, and let $\fp_0\colonequals \fp(0)$ be the constant term of $\fp$.
Assume $\fp\neq T$, so $\fp_0\neq 0$. Then $\fp(\pi)$ is separable, $\phi[\fp]=\ker\fp(\pi)\cong \F_\fp^n$, and $\pi=\phi_T$ acts on $\phi[\fp]$
as multiplication by $\theta\colonequals T\bmod \fp$. We have a natural homomorphism
$$
\iota_\fp\colon R=\End(\phi)\To \End_{\F_\fp}(\phi[\fp])\cong \Mat_n(\F_\fp).
$$
The kernel of $\iota_\fp$ is the two-sided ideal $R\fp(\pi)$ by \cite[Lem. 3.1.16]{PapikianGTM}, and comparing $\F_q$-dimensions
shows that $\iota_\fp$ induces an isomorphism of $\F_q$-algebras
\begin{equation}\label{eq_ipquat}
R_\fp\colonequals R/R\fp(\pi)\xrightarrow{\ \sim\ } \End_{\F_\fp}(\phi[\fp])\cong \Mat_n(\F_\fp).
\end{equation}

For $0\neq u\in R$, let
$$
V_u\colonequals \ker(u)\cap \phi[\fp]=\ker \iota_\fp(u).
$$
This is an $\F_\fp$-subspace of $\phi[\fp]$, since $u$ commutes with $\phi_T$. Let $\rank_\fp(u)$ be the rank of $\iota_\fp(u)$, so that
$\dim_{\F_\fp}V_u=n-\rank_\fp(u)$.

\begin{prop}\label{propExtremal}
Let $1\leq m\leq n$ and let $0\neq u\in R$ with $\deg_\tau(u)\leq dm$. If $\dim_{\F_\fp}V_u\geq m$, 
then:
\begin{enumerate}
\item $\deg_\tau(u)=dm$, $\ker(u)=V_u$, $\dim_{\F_\fp}V_u=m$, and $u_0\neq 0$;
\item $\overline{P}_2(T)=\fp(T)^m$, i.e., the characteristic polynomial of $\pi$ acting on $\ker(u)$ is $\fp(T)^m$;
\item $N(u_0/u_{dm})=(-1)^{dm(n-1)}\fp_0^m$.
\end{enumerate}
\end{prop}
\begin{proof}
(1) We have $q^{dm}\leq \#V_u\leq \#\ker(u)\leq q^{\deg_\tau(u)}\leq q^{dm}$. Hence all are equalities, so $\ker(u)=V_u$,
$\deg_\tau(u)=dm$, and $u$ has $q^{\deg_\tau(u)}$ distinct roots, so $u$ is separable and $u_0\neq 0$.

(2) By (1) and Lemma \ref{lemKer}, $\overline{P}_2(T)$ is the characteristic polynomial of $\pi$ acting on $\ker(u)=V_u\subseteq \phi[\fp]$.
On this space $\pi$ satisfies $\fp(\pi)=0$. The characteristic polynomial over $\F_q$ of an operator
annihilated by the irreducible polynomial $\fp$ is a power of $\fp$, and it has degree $\dim_{\F_q}V_u=dm$.

(3) By (2), $\overline{P}_2(0)=\fp_0^m$. Now apply Proposition \ref{propNorm} with $s=dm$.
\end{proof}

Let $\rho\in \Aut(k)$, and let $\F_{q_0}$ be the subfield of $\F_q$ fixed by $\rho$.
For $1\leq m\leq n$ and $\eta\in k$ put
$$
\cM_m(\eta,\rho)=\{u=u_0+u_1\tau+\cdots+u_{dm}\tau^{dm}\in R\mid u_{dm}=\eta\rho(u_0)\}.
$$
This is an $\F_{q_0}$-subspace of $R$ of dimension $ndm[\F_q:\F_{q_0}]$.

\begin{thm}\label{thmSheekey}
If
\begin{equation}\label{eqCond}
\Nr_{k/\F_{q_0}}(\eta)\cdot \Nr_{\F_q/\F_{q_0}}\left((-1)^{dm(n-1)}\fp_0^m\right)\neq 1,
\end{equation}
then $\iota_\fp$ is injective on $\cM_m(\eta,\rho)$, and $\iota_\fp(\cM_m(\eta,\rho))\subseteq \End_{\F_\fp}(\phi[\fp])\cong \Mat_n(\F_\fp)$
is an $\F_{q_0}$-linear MRD code with minimum distance $n-m+1$. In particular, if $m=1$ and $\rho$ fixes $\F_q$, then the image of
$\cM_1(\eta, \rho)$ in $\End_{\F_\fp}(\phi[\fp])$ is a semifield code.
\end{thm}
\begin{proof}
Let $0\neq u\in \cM_m(\eta, \rho)$. Since $\deg_\tau(u)\leq dm$, we have $\#V_u\leq q^{dm}$, so $\dim_{\F_\fp}V_u\leq m$. Suppose
$\dim_{\F_\fp}V_u=m$. By Proposition \ref{propExtremal}, $u_0\neq 0$, $u_{dm}=\eta\rho(u_0)\neq 0$, and
$$
\frac{N(u_0)}{N(\eta)\, N(\rho(u_0))}=c\colonequals (-1)^{dm(n-1)}\fp_0^m.
$$
Since $\Aut(k)$ is abelian, $N(\rho(u_0))=\rho(N(u_0))$. Put $y=N(u_0)\in \F_q^\times$. Then
$y/\rho(y)=c\,N(\eta)$. The restriction of $\rho$ to $\F_q$ lies in $\Gal(\F_q/\F_{q_0})$, so applying $\Nr_{\F_q/\F_{q_0}}$
gives $1=\Nr_{\F_q/\F_{q_0}}(c)\Nr_{k/\F_{q_0}}(\eta)$, contradicting \eqref{eqCond}. Hence every nonzero $u\in \cM_m(\eta,\rho)$
has $\dim_{\F_\fp}V_u\leq m-1$, i.e., $\rank_\fp(u)\geq n-m+1$. In particular, $\iota_\fp$ is injective on $\cM_m(\eta, \rho)$.
Finally, $\#\cM_m(\eta,\rho)=q^{ndm}=(q^d)^{n(n-(n-m+1)+1)}$, which is the bound \eqref{eqSingleton} for $n\times n$ matrices over $\F_\fp$
with minimum distance $n-m+1$. If $m=1$ and $\rho$ fixes $\F_q$, then $\cM_1(\eta, \rho)$ is an $\F_q$-vector space of dimension $nd$
whose nonzero elements act invertibly on $\phi[\fp]$, so its image is a semifield code by \eqref{eqSFdim}.
\end{proof}

\begin{rem}\label{remSpecial} \hfill
\begin{enumerate}
\item For $m=1$ and $\rho=\id$, condition \eqref{eqCond} reads $N(\eta)\cdot (-1)^{d(n-1)}\fp_0\neq 1$, and the message space is
$\{u_0+u_1\tau+\cdots+u_d\tau^d\in R\mid u_d=\eta u_0\}$.
\item For $\eta=0$ we get the codes $\{u\in R\mid \deg_\tau(u)<dm\}$, which generalize Example \ref{BasicExample}. 
For $m=1$ these are the semifield codes of cyclic semifields (Example \ref{exPetit}). By \cite{SZ}, for $n, d\geq 2$ the cyclic semifields obtained 
from two primes $\fp$ and $\fp'$ of degree $d$ are isotopic if and only if $\fp'(T)=\lambda^d\fp^{\varsigma}(\lambda^{-1}T)$ for some 
$\lambda\in \F_q^\times$ and $\varsigma\in \Aut(\F_q)$, where $\fp^\varsigma$ is obtained by applying $\varsigma$ to the coefficients of $\fp$, 
or $n=2$ and the same holds with $\fp$ replaced by its monic reciprocal polynomial $\fp_0^{-1}T^d\fp(T^{-1})$.
\item For $\fp=T-1$ we get the twisted Gabidulin codes, and their generalizations with $\rho$ not fixing $\F_q$. For $\rho=\id$ on $\F_q$,
condition \eqref{eqCond} becomes $N(\eta)\neq (-1)^{nm}$.
\item In general, $\cM_m(\eta, \rho)$ is the code $S_{n,d,m}(\eta,\rho,\fp)$ of \cite[Thm. 7]{Sheekey} for $\sigma=\Frob_q$,
and \eqref{eqCond} is the condition of \cite[Thm. 7]{Sheekey}. Thus Theorem \ref{thmSheekey} recovers \cite[Thm. 7]{Sheekey} for $\sigma=\Frob_q$.
\end{enumerate}
\end{rem}


\section{New semifields of order $q^9$}\label{secHigher}

In Theorem \ref{thmSheekey} with $m=1$ the message space consists of elements of $R$ of $\tau$-degree at most $\deg(\fp)$.
The proof is short because an endomorphism $u$ of degree $\deg(\fp)$ that does not act invertibly on $\phi[\fp]$ has kernel
equal to an $\F_\fp$-line of $\phi[\fp]$ (Proposition \ref{propExtremal}). On this line $\pi=\phi_T$ acts as multiplication by $T\bmod\fp$, so its characteristic polynomial 
as an $\F_q$-linear map of $\ker(u)$ is $\fp(T)$, and Proposition \ref{propNorm} relates the constant term $\fp_0$ of this polynomial to $N(u_0/u_s)$.

In this section we keep the Drinfeld module $\phi_T=\pi$ of Section \ref{secSheekey}, but allow $\cM$ to contain elements of larger degree.
We describe a family of message spaces for which the left idealizer is $k$ and the right idealizer is, in general, only $\F_q$.
For $n=\deg(\fp)=3$ we construct, for every $q$ and every prime $\fp$ of degree $3$, semifield codes in this family with nuclear parameters
$(9; 3, 1, 3; 1)$, using endomorphisms of degree $\deg(\fp)+2=5$ (Theorem \ref{thmq9} and Proposition \ref{propNucq9}).
If such an endomorphism $u$ is not invertible on $\phi[\fp]$, its kernel still contains an $\F_\fp$-line, but the Frobenius now has two further
eigenvalues on $\ker(u)$, which we do not control. The point is that the next coefficients at the two ends of $\det(U)$, which are the traces
of $\pi$ and of $\pi^{-1}$ on $\ker(u)$, can be combined so that these two eigenvalues disappear. The twist in the message space is exactly
what makes this combination independent of the free coefficients of $u$.
We also report on exhaustive computations for $q\in \{2, 3\}$, which show that the family $\cM_V$ contains many more asymmetric semifield codes than those of Theorem \ref{thmq9} (Example \ref{exComp}).
To the best of our knowledge, no semifields with these nuclear parameters
(or with nuclear parameters in the same Knuth orbit) were previously known; see $\S$\ref{ssNovelty}.

Throughout this section we use the notation of Section \ref{secSheekey}: $k=\F_{q^n}$, $R=\twist{k}$, $\pi=\tau^n$, $\phi_T=\pi$,
and $N=\Nr_{k/\F_q}$. Let $\fp\neq T$ be a monic prime of $A$ of degree $n$, and let $\theta\colonequals T\bmod \fp\in \F_\fp$,
so that $\F_\fp=\F_q[\theta]$ and $\pi=\phi_T$ acts on $\phi[\fp]$ as multiplication by $\theta$.
By \eqref{eq_ipquat}, $\iota_\fp$ induces an isomorphism of $\F_q$-algebras
$$
R_\fp= R/R\phi_\fp=R/R\fp(\pi)\xrightarrow{\ \sim\ } \End_{\F_\fp}(\phi[\fp])\cong \Mat_n(\F_\fp),
$$
and every element of $R_\fp$ has a unique representative in $R$ of $\tau$-degree $<n^2$.

\subsection{Message spaces stable under $k$}\label{ssFamily}

Since $\deg(\fp)=n$, we have $\phi[\fp]\cong \F_\fp^n$, a semifield code in $\End_{\F_\fp}(\phi[\fp])$ has $\F_q$-dimension $n^2$,
and the corresponding semifields have order $q^{n^2}$.
Besides $k$, the ring $\End(\phi)=R$ contains the subring $\F_q\{\tau\}$. Let
$$
\cF_\fp\colonequals \F_q\{\tau\}/(\phi_\fp)=\F_q\{\tau\}/(\fp(\tau^n))\subset R_\fp.
$$
This is a commutative $\F_q$-algebra of dimension $n^2$ containing $\F_\fp=\F_q[\pi]/(\fp(\pi))$, and $\cF_\fp=\F_\fp[\tau]$
has dimension $n$ over $\F_\fp$, so the minimal polynomial of $\tau$ over $\F_\fp$ is $x^n-\theta$.
For an $\F_q$-subspace $V\subseteq \cF_\fp$ with $1\in V$ and $\dim_{\F_q}V=n$, let
\begin{equation}\label{eqcMU}
\cM_V\colonequals k\cdot V=\left\{\sum_{v\in V} a_v v\mid a_v\in k\right\}\subseteq R_\fp.
\end{equation}
Then $\cM_V$ is a left $k$-subspace of $R_\fp$ containing $k$. The natural map $k\otimes_{\F_q}\cF_\fp\to R_\fp$, $a\otimes v\mapsto av$, 
is an isomorphism of $\F_q$-vector spaces, since $1, \tau, \dots, \tau^{n^2-1}$ is an $\F_q$-basis of $\cF_\fp$ and a left $k$-basis of $R_\fp$. 
Hence $\dim_{\F_q}\cM_V=n^2$ and $\cM_V\cap \cF_\fp=V$, so distinct subspaces $V$ give distinct codes $\cM_V$. The case $V=\F_q\oplus \F_q\tau\oplus\cdots\oplus \F_q\tau^{n-1}$ gives
the semifield code $\cM=\{u\in R\mid \deg_\tau(u)\leq n-1\}$ of Example \ref{BasicExample}, whose semifields are the cyclic semifields of Petit \cite{Petit}. 

\begin{lem}\label{lemCentk}
	The centralizer of $k$ in $R_\fp$ is the image of $k[\pi]$, which is isomorphic to $k\otimes_{\F_q}\F_\fp$. Every subfield of $R_\fp$ 
	contained in this centralizer has degree at most $n$ over $\F_q$.
\end{lem}
\begin{proof}
	An element $\tau^i$ commutes with all of $k$ if and only if $n\mid i$, so the centralizer of $k$ in $R$ is $k[\pi]$. If $w\in R$ has 
	degree less than $n^2=\deg_\tau\fp(\pi)$ and its image in $R_\fp$ commutes with $k$, then for every $\lambda\in k$ the element 
	$\lambda w-w\lambda$ has degree less than $n^2$ and lies in $R\fp(\pi)$, so $\lambda w=w\lambda$; hence $w\in k[\pi]$. The image of 
	$k[\pi]$ in $R_\fp$ is $k[\pi]/(\fp(\pi))\cong k\otimes_{\F_q}\F_\fp$. Since $k$ and $\F_\fp$ are both isomorphic to $\F_{q^n}$, we have 
	$k\otimes_{\F_q}\F_\fp\cong \F_{q^n}^n$, a product of $n$ fields. If $E$ is a subfield of $\F_{q^n}^n$, then each projection restricts 
	to an injective homomorphism $E\to \F_{q^n}$, so $[E:\F_q]\leq n$.
\end{proof}

\begin{lem}\label{lemNucHigher} Let $\cM\subseteq R_\fp$ be a left $k$-subspace with $1\in \cM$ and
	$\dim_{\F_q}\cM=n^2$, and assume that $\cM$ is a semifield code. Then
	$$
	I_l(\cM)=k, \qquad C(\cM)=\F_\fp, \qquad Z(\cM)=\F_q,
	$$
	where $k$ acts on $\phi[\fp]$ through $R_\fp$ and $\F_\fp$ acts through $\phi$, and $[I_r(\cM):\F_q]$ is a proper divisor of $n^2$. 
	In particular, if $n$ is prime, then the nuclear parameters of $\cM$ are $(n^2; n, 1, n; 1)$ or $(n^2; n, n, n; 1)$.
\end{lem}
\begin{proof}
	By Proposition \ref{lemCompare}, $I_l(\cM), I_r(\cM)\subseteq \cM\subseteq R_\fp$ and $C(\cM)=C^\fp(\cM)\subseteq R_\fp$, and these are fields. 
	Since $1\in\cM$ and $\cM$ is a left $k$-subspace, $k\subseteq I_l(\cM)$. As $I_l(\cM)$ is a field containing $k$, its elements 
	commute with $k$, so $[I_l(\cM):\F_q]\leq n$ by Lemma \ref{lemCentk}, and $I_l(\cM)=k$.

	The elements of $\F_\fp$ commute with $\End_{\F_\fp}(\phi[\fp])\supseteq\cM$, so $\F_\fp\subseteq C(\cM)$. Since $k\subseteq \cM$, the 
	elements of $C(\cM)$ commute with $k$, so $[C(\cM):\F_q]\leq n$ by Lemma \ref{lemCentk}, and $C(\cM)=\F_\fp$.

	Next, $\cM$ is not commutative: otherwise $\cM$ would be contained in the centralizer of $k\subseteq \cM$, which has $\F_q$-dimension $n^2$ 
	by Lemma \ref{lemCentk}, so $\cM$ would be equal to $k\otimes_{\F_q}\F_\fp\cong \F_{q^n}^n$, which has zero-divisors. The field $I_r(\cM)$ 
	acts on $\phi[\fp]$, so $[I_r(\cM):\F_q]$ divides $n^2$, and if it were equal to $n^2$, then $I_r(\cM)=\cM$ would be a field, which 
	is impossible.

	Finally, $Z(\cM)=I_l(\cM)\cap C(\cM)=k\cap \F_\fp$. Since $\F_\fp$ is the center of $\End_{\F_\fp}(\phi[\fp])$, an element $\lambda\in k$ 
	lies in $\F_\fp$ if and only if it commutes with $\tau$. As $\tau\lambda-\lambda\tau=(\lambda^q-\lambda)\tau$ and $\tau$ acts invertibly 
	on $\phi[\fp]$, this happens if and only if $\lambda\in \F_q$. Thus $Z(\cM)=\F_q$.
\end{proof}

\begin{rem} The codes $\cM_V$ of \eqref{eqcMU} are left $k$-subspaces of $R_\fp$ of dimension $n^2$ containing $1$, so by Lemma \ref{lemNucHigher}, 
	for $n$ prime each semifield code among them has nuclear parameters $(n^2; n, 1, n; 1)$ or $(n^2; n, n, n; 1)$.
	The second possibility occurs, for example, when 
	$V=\F_q\oplus\F_q t_1\tau \oplus\cdots\oplus \F_q t_{n-1}\tau^{n-1}$ with $t_i\in \F_\fp$: then $t_i\tau^i\lambda=t_i\lambda^{q^i}\tau^i$ 
	for $\lambda\in k$, so $\cM_V$ is also a right $k$-subspace and $k\subseteq I_r(\cM_V)$; this includes Petit's code. 
	We say that a semifield code $\cM$ as in Lemma \ref{lemNucHigher} is \textit{asymmetric} if $I_r(\cM)=\F_q$.
\end{rem}

\subsection{Frobenius on the kernel of an endomorphism}\label{ssFrobKer}

From now on $n=3$, so $k=\F_{q^3}$, $\pi=\tau^3$, and we write
$$
\fp=T^3+\fp_2T^2+\fp_1T+\fp_0\in A, \qquad \fp_0\neq 0.
$$
Then $\Nr_{\F_\fp/\F_q}(\theta)=-\fp_0$. Let $\Tr=\Tr_{k/\F_q}$. For $f=\sum_i f_iT^i\in k[T]$ we put $N(f)=ff^{(q)}f^{(q^2)}$ and
$\Tr(f)=f+f^{(q)}+f^{(q^2)}$, with $f^{(q^i)}$ as in $\S$\ref{secMotive}; these are the norm and the trace for the extension $k[T]/\F_q[T]$, and they lie in $\F_q[T]$.

Let $u=u_0+u_1\tau+\cdots+u_s\tau^s\in R$ with $u_0\neq 0$ and $u_s\neq 0$, and let $U\in \Mat_3(k[T])$ be the matrix by which $u$ acts on
the motive $M_\phi$ with respect to the standard basis $1, \tau, \tau^2$ ($\S$\ref{secMotive}). Let
$$
\chi_u(T)\colonequals \overline{P}_2(T)\in \F_q[T]
$$
be the characteristic polynomial of $\pi$ acting on the $s$-dimensional $\F_q$-vector space $\ker(u)$, as in Section \ref{secSheekey}.

Since $T$ acts on $M_\phi$ as right multiplication by $\pi=\tau^3$, we can write $u=A+B\tau+C\tau^2$ in $M_\phi$ with $A, B, C\in k[T]$, namely
$A=\sum_i u_{3i}T^i$, $B=\sum_i u_{3i+1}T^i$, $C=\sum_iu_{3i+2}T^i$. The coordinates of $u$, $\tau u$, $\tau^2u$ in the standard basis
are $(A,B,C)$, $(TC^{(q)},A^{(q)},B^{(q)})$ and $(TB^{(q^2)},TC^{(q^2)},A^{(q^2)})$, so
\begin{equation}\label{s4eqU}
\det(U)=N(A)+T\cdot N(B)+T^2\cdot N(C)-T\cdot\Tr\left(A\,B^{(q)}\,C^{(q^2)}\right).
\end{equation}
(For $C=0$ this is the case $r=3$, $s=1$, $t=0$ of the computation in Example \ref{ssExample}.) In particular, $\det(U)\equiv N(u_0)\pmod T$.
On the other hand, by \eqref{eqPbar}, $\det(U)=c\cdot \chi_u(T)$ for some $c\in \F_q^\times$, and by Proposition \ref{propNorm}, $\chi_u(0)=N(u_0/u_s)$.
Comparing the constant terms of these two expressions for $\det(U)$, we get $c=N(u_s)$, i.e.,
\begin{equation}\label{s4eqDet}
\det(U)=N(u_s)\cdot \chi_u(T).
\end{equation}
By Corollary \ref{corInv},
\begin{equation}\label{s4eqCrit}
u \text{ does not act invertibly on } \phi[\fp] \quad\Longleftrightarrow\quad \fp \mid \chi_u.
\end{equation}

\begin{lem}\label{s4lemEnds}
Let $u=\sum_{j=0}^5u_j\tau^j\in R$ and write $\det(U)=\sum_{i=0}^5 n_iT^i$. Then
\begin{align*}
n_0&=N(u_0), &
n_1&=N(u_1)+\Tr\left(u_3\,u_0^q\,u_0^{q^2}\right)-\Tr\left(u_0\,u_1^q\,u_2^{q^2}\right),\\
n_5&=N(u_5), &
n_4&=N(u_4)+\Tr\left(u_2\,u_5^q\,u_5^{q^2}\right)-\Tr\left(u_3\,u_4^q\,u_5^{q^2}\right).
\end{align*}
Consequently, if $u_0\neq 0\neq u_5$, then
$$
\det(\pi\mid\ker u)=-\frac{n_0}{n_5},\qquad
\Tr(\pi\mid\ker u)=-\frac{n_4}{n_5},\qquad
\Tr(\pi^{-1}\mid\ker u)=-\frac{n_1}{n_0}.
$$
\end{lem}
\begin{proof}
Put $A=u_0+u_3T$, $B=u_1+u_4T$, $C=u_2+u_5T$ in \eqref{s4eqU}. The terms of $T$-degree $0$
and $1$ come from $N(A)$, $TN(B)$ and the trace term, and the terms of degree $4$ and $5$ from
$TN(B)$, $T^2N(C)$ and the trace term; collecting them gives the formulas. The second part
follows from \eqref{s4eqDet}: if $\chi_u=\prod_{i=1}^5(T-\lambda_i)$, then the constant term of
$\chi_u$ is $-\prod\lambda_i$, the coefficient of $T^4$ is $-\sum\lambda_i$, and the coefficient of
$T$ divided by the constant term is $-\sum\lambda_i^{-1}$.
\end{proof}

\begin{thm}\label{thmq9}
Let $\beta,\kappa\in\F_q^\times$ with $\beta\kappa^2=-\fp_0$; for example, $\kappa=1$ and $\beta=-\fp_0$.
Put
\begin{align*}
&g_{\beta,\kappa}\colonequals(\tau+\beta)(\tau^4+\kappa)=\beta\kappa+\kappa\tau+\beta\tau^4+\tau^5\in\F_q\{\tau\}, \\ 
&\cM_{\beta,\kappa}\colonequals k\,g_{\beta,\kappa}\oplus k\tau^2\oplus k\tau^3\subset \End(\phi).
\end{align*}
Every nonzero $u\in\cM_{\beta,\kappa}$ acts invertibly on $\phi[\fp]$. Since
$\dim_{\F_q}\cM_{\beta,\kappa}=9=3\deg(\fp)$, the image of $\cM_{\beta,\kappa}$ in
$\End_{\F_\fp}(\phi[\fp])$ is a semifield code.
\end{thm}

\begin{proof}
Let $0\neq u=c\,g_{\beta,\kappa}+a\tau^2+b\tau^3$ with $a,b,c\in k$.

If $c=0$, then $u=(a+b\tau)\tau^2$. The map $\tau^2$ is an $\F_\fp$-linear automorphism of $\phi[\fp]$, since $\tau$ commutes with $\phi_T$.
So a nonzero element of $\ker(u)\cap\phi[\fp]$ would give an $\F_\fp$-line, with $q^3$ elements, inside the kernel of
$a+b\tau$, which has at most $q$ elements. Hence $u$ acts invertibly on $\phi[\fp]$, as in
Example \ref{BasicExample}.

Let $c\neq 0$. Replacing $u$ by $c^{-1}u$, we may assume $c=1$, so
$$
u=\beta\kappa+\kappa\tau+a\tau^2+b\tau^3+\beta\tau^4+\tau^5,\qquad u_0=\beta\kappa\neq0,\quad u_5=1.
$$
By Lemma \ref{s4lemEnds},
\begin{align}\label{s4eqN}
& n_0=\beta^3\kappa^3,\qquad n_5=1,\qquad n_4=\beta^3+L,\qquad n_1=\kappa^3-\beta\kappa^2L, \\ 
\nonumber & \text{where } L\colonequals \Tr(a)-\beta\Tr(b).
\end{align}
Suppose $u$ does not act invertibly on $\phi[\fp]$. By \eqref{s4eqCrit}, $\fp\mid\chi_u$; equivalently,
$\ker(u)\cap\phi[\fp]$ is a nonzero $\F_\fp$-subspace, on which $\pi$ acts as $\theta$. Hence
the eigenvalues of $\pi$ on the five-dimensional space $\ker(u)$ are
$$
\theta,\ \theta^q,\ \theta^{q^2},\ \mu_1,\ \mu_2
$$
for some $\mu_1,\mu_2\in\overline{\F}_q^\times$, about which we know nothing a priori. 
For $w\in\F_q$, 
consider
$$
\Sigma_w\colonequals\Tr(\pi\mid\ker u)+w\,\Tr(\pi^{-1}\mid\ker u).
$$
We compute $\Sigma_w$ in two ways.

\emph{Through the eigenvalues.} The roots of $\fp$ contribute $\sum_i\theta^{q^i}=-\fp_2$ and
$\sum_i\theta^{-q^i}=-\fp_1/\fp_0$. By Lemma \ref{s4lemEnds}, $\det(\pi\mid\ker u)=-n_0$, so
$\mu_1\mu_2=-n_0/\Nr_{\F_\fp/\F_q}(\theta)$. Therefore
\begin{equation}\label{s4eqSide1}
\Sigma_w=-\fp_2-\frac{w\fp_1}{\fp_0}+(\mu_1+\mu_2)\left(1-\frac{w\,\Nr_{\F_\fp/\F_q}(\theta)}{n_0}\right).
\end{equation}

\emph{Through the coefficients of $u$.} By Lemma \ref{s4lemEnds} and \eqref{s4eqN},
\begin{equation}\label{s4eqSide2}
\Sigma_w=-n_4-\frac{w\,n_1}{n_0}=-\beta^3-\frac{w\kappa^3}{n_0}+L\left(\frac{w\,\beta\kappa^2}{n_0}-1\right).
\end{equation}

Now choose $w=n_0/\Nr_{\F_\fp/\F_q}(\theta)$. This kills the unknown eigenvalues in
\eqref{s4eqSide1}. Since $\Nr_{\F_\fp/\F_q}(\theta)=-\fp_0=\beta\kappa^2$, the same $w$ also kills the
coefficient of $L$ in \eqref{s4eqSide2}, so the free coefficients $a, b$, which enter only through $L$, drop out. Explicitly, $w=\beta^3\kappa^3/(\beta\kappa^2)=\beta^2\kappa$,
and equating the two expressions gives
$$
-\fp_2+\frac{\beta \fp_1}{\kappa}=-\beta^3-\frac{\kappa}{\beta}.
$$
Multiplying by $-\kappa^2/\beta^2$ and using $\fp_0=-\beta\kappa^2$, this says that $\fp(-\kappa/\beta)=0$.
This is impossible, because $\fp$ is irreducible of degree $3$ and $-\kappa/\beta\in\F_q$.
\end{proof}

\begin{prop}\label{propNucq9}
The semifield code of Theorem \ref{thmq9} has nuclear parameters $(9;3,1,3;1)$.
\end{prop}
\begin{proof}
We apply Definition \ref{defNPq} with $X=\tau^2\in \cM_{\beta,\kappa}$. In $R_\fp$ we have $\tau^3=\pi=T$, where $T$ denotes $\pi=\phi_T$,
so $\tau$ is invertible in $R_\fp$ with $\tau^{-1}=T^{-1}\tau^2$. Hence the nuclear parameters of $\cM_{\beta,\kappa}$ are those of
$\cM^\natural\colonequals\cM_{\beta,\kappa}\tau^{-2}\subset R_\fp$, and
$$
\cM^\natural=k\oplus k\tau\oplus k\,w_0,\qquad
w_0=g_{\beta,\kappa}\tau^{-2}=T+\beta\kappa T^{-1}\tau+(\beta+\kappa T^{-1})\tau^2 .
$$
Thus $\cM^\natural=\cM_V$ with $V=\langle 1,\tau,w_0\rangle_{\F_q}\subset\cF_\fp$, and by
Lemma \ref{lemNucHigher} it remains to show that $I_r(\cM^\natural)=\F_q$.

Every element of $R_\fp$ can be written uniquely as $x_0+x_1\tau+x_2\tau^2$ with
$x_j\in k[T]/(\fp)=k\oplus kT\oplus kT^{-1}$; here we use that $1,\theta,\theta^{-1}$ is an
$\F_q$-basis of $\F_\fp$. Let $X=a+b\tau+dw_0\in I_r(\cM^\natural)\subseteq\cM^\natural$. From
$$
\tau w_0=(\beta T+\kappa)+T\tau+\beta\kappa T^{-1}\tau^2,
$$
the $kT$-component of the $\tau$-coefficient of $\tau X$ is $d^q$, while for every element of
$\cM^\natural$ this component is $0$. Since $\tau X\in\cM^\natural$, we get $d=0$. Then
$\tau X=a^q\tau+b^q\tau^2$. If $\tau X=a'+b'\tau+d'w_0$, the $kT^{-1}$-component of the
$\tau^2$-coefficient gives $\kappa d'=0$, so $d'=0$, and then $b^q=\beta d'=0$, so $b=0$. Finally,
$w_0a=Ta+\beta\kappa T^{-1}a^q\tau+(\beta+\kappa T^{-1})a^{q^2}\tau^2\in\cM^\natural$ forces
$a^q=a$, i.e., $a\in\F_q$.
\end{proof}

The method of proof of Theorem \ref{thmq9} applies to other message spaces. The next proposition gives a second family of semifield codes 
with the same nuclear parameters, using endomorphisms of degree $4$; it requires a condition on $\fp$, but its proof is shorter. 

\begin{prop}\label{propDeg4}
Assume $\fp_2=-1$, i.e., $\fp=T^3-T^2+\fp_1T+\fp_0$, and put
$$
\cM'\colonequals k\oplus k(\tau-\tau^2)\oplus k(\tau^3-\tau^4)\subset \End(\phi).
$$
Every nonzero $u\in \cM'$ acts invertibly on $\phi[\fp]$, and the image of $\cM'$ in $\End_{\F_\fp}(\phi[\fp])$ is a semifield code 
with nuclear parameters $(9; 3, 1, 3; 1)$.
\end{prop}
\begin{proof}
Let $0\neq u=a+b(\tau-\tau^2)+c(\tau^3-\tau^4)$ with $a, b, c\in k$. If $c=0$, then $\deg_\tau(u)\leq 2$, so $u$ acts invertibly on 
$\phi[\fp]$ by Example \ref{BasicExample}. Let $c\neq 0$. Replacing $u$ by $c^{-1}u$, we may assume $c=1$. The formula \eqref{s4eqU} for $\det(U)$ holds for every $u\in R$. 
In our case $A=a+T$, $B=b-T$, $C=-b$, and a direct computation gives
$$
\det(U)=-T^4+T^3+(\text{terms of degree}\leq 2), \qquad \det(U)\equiv N(a) \pmod T;
$$
here the two terms $\pm\Tr(b)T^3$, coming from $T\cdot N(B)$ and from the trace term, cancel. Suppose $u$ does not act invertibly on $\phi[\fp]$. 
By Corollary \ref{corInv}, $\fp$ divides $\det(U)$, so $\det(U)=(-T+c_0)\fp$ for some $c_0\in \F_q$. Comparing the coefficients of $T^3$ and using 
$\fp_2=-1$, we get $c_0=0$. Hence $N(a)=0$, so $a=0$ and
$$
u=b(\tau-\tau^2)+\tau^3-\tau^4=(b+\tau^2)\,\tau\,(1-\tau).
$$
Each of the three factors is nonzero of degree at most $2$, so it acts invertibly on $\phi[\fp]$ by Example \ref{BasicExample}, a contradiction. 
Since $\dim_{\F_q}\cM'=9$, the image of $\cM'$ is a semifield code.

We have $\cM'=\cM_V$ with $V=\langle 1, v_1, v_2\rangle_{\F_q}$, where $v_1=\tau-\tau^2$ and $v_2=\tau^3-\tau^4=T-T\tau$ in $R_\fp$. 
By Lemma \ref{lemNucHigher}, it remains to show that $I_r(\cM')=\F_q$. As in the proof of Proposition \ref{propNucq9}, we write the elements of $R_\fp$ as 
$x_0+x_1\tau+x_2\tau^2$ with $x_j\in k[T]/(\fp)=k\oplus kT\oplus kT^2$. The elements of $\cM'$ are those of the form 
$(a'+c'T)+(b'-c'T)\tau-b'\tau^2$ with $a', b', c'\in k$. Let $X=a+bv_1+cv_2\in I_r(\cM')\subseteq \cM'$. Then
$$
v_1X=\left(-(b^q+b^{q^2})T+c^{q^2}T^2\right)+\left(a^q+(c^q+b^{q^2})T\right)\tau+\left(b^q-a^{q^2}-(c^q+c^{q^2})T\right)\tau^2
$$
lies in $\cM'$. Comparing this with $(a'+c'T)+(b'-c'T)\tau-b'\tau^2$, the $kT^2$-component of the constant coefficient gives $c=0$. 
The $kT$-components of the constant coefficient and of the $\tau$-coefficient then give $c'=-(b^q+b^{q^2})$ and $-c'=b^{q^2}$, so $b=0$. 
Finally, the $\tau$- and $\tau^2$-coefficients give $b'=a^q$ and $b'=a^{q^2}$, so $a\in \F_q$.
\end{proof}

\begin{rem}\label{remDeg4}
Primes $\fp$ with $\fp_2=-1$ exist for every $q$: there are $q^2$ elements of $\F_{q^3}$ of trace $1$ over $\F_q$, at most one of them lies 
in $\F_q$, and the minimal polynomial over $\F_q$ of any other one is such a prime. For $q=2$ and $\fp=T^3+T^2+1$ the semifield codes of
Theorem \ref{thmq9} and of Proposition \ref{propDeg4} are not equivalent; see Example \ref{exComp}. On the other hand, the code of
Proposition \ref{propDeg4} for a given $\fp$ can be equivalent to a code of Theorem \ref{thmq9} for a different prime. A computer search
shows that for $q=2$ the code of Proposition \ref{propDeg4} for $T^3+T^2+1$ is equivalent to the code of Theorem \ref{thmq9} for $T^3+T+1$,
and that for $q=3$ each of the three codes of Proposition \ref{propDeg4} is equivalent to a code of Theorem \ref{thmq9} for some other
prime. We do not know whether this holds for every $q$.
\end{rem}

\subsection{Explicit computations}\label{ssComp}

We describe how the codes $\cM_V$ can be computed. Fix a field isomorphism $\xi\colon k\to \F_\fp$.
The $\F_\fp$-linear operator $\tau$ on $\phi[\fp]$ has minimal polynomial $x^3-\theta$ over $\F_\fp$ (see $\S$\ref{ssFamily}),
and $k$ acts on $\phi[\fp]$ through the three characters $\lambda\mapsto \xi(\lambda)^{q^i}$, $i=0, 1, 2$,
with one-dimensional eigenspaces permuted by $\tau$ (because $\tau\lambda=\lambda^q\tau$). Hence there is an $\F_\fp$-basis
$e_0, e_1, e_2$ of $\phi[\fp]$ such that
$$
\tau e_0=e_1, \quad \tau e_1=e_2, \quad \tau e_2=\theta e_0, \qquad \lambda e_i=\xi(\lambda)^{q^{-i}}e_i \quad (\lambda\in k).
$$
With respect to this basis,
$$
\iota_\fp(\tau)=\begin{bmatrix} 0 & 0 & \theta\\ 1 & 0 & 0\\ 0 & 1 & 0\end{bmatrix}, \qquad
D(\lambda)\colonequals\iota_\fp(\lambda)=\diag\left(\xi(\lambda), \xi(\lambda)^{q^2}, \xi(\lambda)^{q}\right),
$$
and $\iota_\fp\colon R_\fp\to \Mat_3(\F_\fp)$ is determined by these two matrices.
If $V=\langle 1, v_1, v_2\rangle_{\F_q}$, then the elements of $\cM_V$ are the matrices
$$
D(a_0)+D(a_1)\iota_\fp(v_1)+D(a_2)\iota_\fp(v_2), \qquad a_0, a_1, a_2\in k.
$$
Since $\det(D(\lambda)M)=\xi(N(\lambda))\det(M)$, the semifield condition
can be checked on representatives of $(k^3\setminus 0)/k^\times$, i.e., on $q^6+q^3+1$ determinants of $3\times 3$ matrices over $\F_\fp$.
For example, for $V=\langle 1, \tau, \tau^2\rangle_{\F_q}$ one finds
\begin{align*}
&\det\left(D(a_0)+D(a_1)\iota_\fp(\tau)+D(a_2)\iota_\fp(\tau)^2\right) \\ &=N(a_0)+\theta N(a_1)+\theta^2N(a_2)-\theta\Tr_{k/\F_q}\left(a_0a_1^qa_2^{q^2}\right). 
\end{align*}
In general, $\det$ is a cubic form in the nine $\F_q$-coordinates of $(a_0, a_1, a_2)$ with values in $\F_\fp$, i.e., a system of
three cubic forms over $\F_q$ in nine variables, and the semifield condition is that this system has no nontrivial zero.

Since every $V$ as in \eqref{eqcMU} has a basis $1, v_1, v_2$ with $v_1, v_2$ in a fixed complement of $\F_q$ in $\cF_\fp$,
the enumeration of all $V$ is finite, and we carried it out for $q=2$ and $q=3$ for one prime $\fp$ each.

Recall that two codes $\cC, \cC'\subseteq \End_{\F_q}(\phi[\fp])$ are equivalent if $\cC'=\{AX^\rho B\mid X\in \cC\}$ for some 
invertible $A, B$ and some $\rho\in \Aut(\F_q)$ \cite[\S 2.5]{Sheekey}, and that two semifields are isotopic if and only if their 
semifield codes are equivalent \cite[\S 2.4]{Sheekey}. For a semifield code $\cM$, a polynomial $f$ with coefficients in the prime field of $\F_q$, and a nonzero $X\in \cM$, 
consider the multiset $\mu_f(X)\colonequals\{\rank f(X^{-1}Y)\mid Y\in \cM\setminus\{0\}\}$. If $\cM'=A\cM^\rho B$, then 
$X'^{-1}Y'=B^{-1}(X^{-1}Y)^\rho B$ for $X'=AX^\rho B$, $Y'=AY^\rho B$, so $\mu_f(X')=\mu_f(X)$. 
Hence the multiset $\{\mu_f(X)\mid X\in \cM\setminus\{0\}\}$ is an invariant of equivalence.

\begin{example}\label{exComp}
The following results were obtained by an exhaustive computer search over all subspaces $V$ as in \eqref{eqcMU},
using the method described above; they are computational and do not follow from Theorem \ref{thmq9}.
	\begin{enumerate}
		\item Let $q=2$ and $\fp=T^3+T^2+1$. Among the codes $\cM_V$ of \eqref{eqcMU} there are exactly $210$ distinct semifield codes:
		$168$ of them are asymmetric, with nuclear parameters $(9; 3, 1, 3; 1)$, and $42$ have nuclear parameters $(9; 3, 3, 3; 1)$.
		The $168$ asymmetric codes fall into at least $8$ equivalence classes, which are distinguished by the above invariant with $f=x^9+x+1$. 
		Exactly one of these codes is given by Theorem \ref{thmq9} (after multiplication by $\tau^{-2}$), the code of Proposition \ref{propDeg4} is another one, 
		and these two codes are not equivalent.
		\item Let $q=3$ and $\fp=T^3+2T^2+1$. Among the codes $\cM_V$ there are exactly $780$ distinct semifield codes: $468$ of them
		are asymmetric, with nuclear parameters $(9; 3, 1, 3; 1)$, and $312$ have nuclear parameters $(9; 3, 3, 3; 1)$.
	\end{enumerate}
	Explicit asymmetric examples are
	\begin{align*}
		q=2:&\qquad \cM_V=k\oplus k(\tau^2+\tau^4+\tau^5+\tau^7)\oplus k(\tau^3+\tau^7), \\
		q=3:&\qquad \cM_V=k\oplus k(2\tau+2\tau^2+\tau^3+\tau^4+2\tau^5+\tau^6+2\tau^7)\oplus k(\tau^4+2\tau^6+\tau^7+2\tau^8).
	\end{align*}
	Neither of these codes is equivalent to a code of Theorem \ref{thmq9} or of Proposition \ref{propDeg4}, for any prime $\fp$
	and any admissible choice of the parameters.
\end{example}

\subsection{Comparison with known semifields}\label{ssNovelty}

By Definition \ref{defNPq} and \cite[Prop. 5]{Sheekey}, a semifield code containing the identity with nuclear parameters
$(N; a, b, c; z)$ arises from a semifield of order $q^N$ whose left, middle and right nuclei and center have
$\F_q$-dimensions $a, b, c, z$, respectively. Nuclear parameters of semifield codes are invariant under equivalence (Definition \ref{defNPq}), or equivalently under
isotopy of semifields, so the codes of Theorem \ref{thmq9} and Example \ref{exComp} give semifields of order $q^9$ whose left and right
nuclei are isomorphic to $\F_{q^3}$ and whose middle nucleus and center are $\F_q$.
The Knuth orbit of a semifield with nuclear parameters
$(N; a, b, c; z)$ consists of semifields whose nuclear parameters are obtained by permuting $(a, b, c)$ \cite{Knuth} (for a semifield, 
we use the $\F_q$-dimensions of the semifield, its left, middle and right nuclei, and its center, in this order). 
Thus, a semifield with nuclear parameters $(9; 3, 1, 3; 1)$ is not isotopic to a semifield whose
nuclear parameters are not a permutation of these.

To the best of our knowledge, the known constructions of semifields with nuclear parameters $(9; a, b, c; 1)$ and
$3\in\{a, b, c\}$ are the following: Albert's generalized twisted fields \cite{Albert}, for which $(a, b, c)$ is then a
permutation of $(1, 1, 3)$; the cyclic semifields of Petit \cite{Petit} and Sandler \cite{Sandler}, with $(a, b, c)=(3, 3, 3)$
\cite[\S 3.1]{Sheekey}; and Sheekey's semifields $S_{3, 3, 1}(\eta, \rho)$ \cite[\S 4.2]{Sheekey}, for which $(a, b, c)$ is a
permutation of $(1, 1, 3)$ if $\eta\neq 0$ and $\rho\neq \mathrm{id}$, and $(a, b, c)=(3, 3, 3)$ if $\eta=0$ or $\rho=\mathrm{id}$.
The other general families of semifields with center $\F_q$ that we are aware of (see the surveys \cite{LP}, \cite{SheekeySurvey},
and \cite{MarPol}) either exist only in even dimension over the center, or have $a=b=c=1$ (e.g., \cite{Kantor}). In none of these
families is $(a, b, c)$ a permutation of $(3, 1, 3)$. Moreover, in a commutative semifield the left and right nuclei coincide and are 
contained in the middle nucleus (see, e.g., \cite{LP}), so $(a, b, c)$ is not a permutation of $(3, 1, 3)$ for any semifield in the 
Knuth orbit of a semifield isotopic to a commutative one.


\section{Infinite families of semifield codes}\label{s2C}

In this section we generalize Example \ref{ssExample} in a different direction from Section \ref{secHigher}:
instead of keeping $\phi_T=\tau^n$, we allow Drinfeld modules with $\phi_T\in \twistt{\F_{q^s}}{\tau^\ell}$,
whose endomorphism rings are smaller. In this more general setting, explicit computations 
of determinants of endomorphisms become unwieldy. Instead, we prove the invertibility of endomorphisms acting on torsion points 
using the duality of Section \ref{secDuality}, the structure of the division algebra $k(\tau)$, and the reduction modulo $\fp$ of the minimal polynomial of the 
Frobenius endomorphism. 

The semifields obtained in this section are cyclic semifields of Petit (Proposition \ref{propPetit}), so they are not new. The purpose 
of the section is to show how the methods of this paper apply to Drinfeld modules whose endomorphism rings are smaller than $\twist{k}$, 
and to describe the resulting semifields in terms of the Frobenius endomorphism $\pi$ of $\phi$: the condition for the message space to give a semifield 
code (Proposition \ref{propCriterion}) and the isotopy class of the semifield (Remark \ref{remSZ}) are both expressed through the minimal polynomial of $\pi$ acting on $\phi[\fp]$. 

\subsection{Division algebras} 

In addition to the notation used in Section \ref{secSheekey}, let $\Delta=k(\tau)$ be the algebra over $K=\F_q(\pi)$ given in terms 
of generators and relations as follows:
\begin{align*}
	\Delta=k(\pi)\oplus k(\pi)\tau\oplus \cdots\oplus k(\pi)\tau^{n-1}, \\ 
	\tau^n=\pi, \quad \tau\alpha=\alpha^q\tau \quad \text{for all}\quad \alpha\in k. 
\end{align*} 
It is known that $\Delta$ is a central division algebra over $K$ of dimension $n^2$ and the local invariants of $\Delta$ are
zero at all places of $K$ except at $\pi$ and $1/\pi$, where the invariants are $1/n$ and $-1/n$, respectively; see \cite[Section 4.1]{PapikianGTM}. 

\begin{defn}
Let $\ell$  and $s$ be integers such that $n=\lcm(\ell, s)$. Let $g=\gcd(\ell,s)$. Denote 
$$
\Delta(\ell, s)=\F_{q^\ell}(\tau^s), \qquad \Delta(s, \ell)=\F_{q^s}(\tau^\ell), 
$$
so that $\Delta(n, 1)=\Delta$ and $\Delta(1, n)=\F_q(\pi)=K$. 
We consider $\Delta(\ell, s)$ and $\Delta(s, \ell)$ as division subalgebras of $\Delta$. 
\end{defn}

\begin{lem} Denote $K_g=\F_{q^g}(\pi)$. 
	\begin{enumerate}
		\item $Z(\Delta(\ell, s))=Z(\Delta(s,\ell))=K_g$, where $Z(\cdot)$ denotes the center of the corresponding algebra. 
		\item $\Cent_\Delta(\Delta(\ell, s))=\Delta(s, \ell)$ and $\Cent_\Delta(\Delta(s,  \ell))=\Delta(\ell, s)$, where $\Cent_\Delta(\cdot)$ denotes the centralizer in $\Delta$ 
		of the corresponding algebra. 
		\item $\dim_{K_g} \Delta(\ell, s) =(\ell/g)^2$ and $\dim_{K_g} \Delta(s, \ell) =(s/g)^2$. 
	\end{enumerate}
\end{lem}
\begin{proof} An element $\alpha=\sum_{j=0}^m a_j \tau^{j s}\in \twistt{\F_{q^\ell}}{\tau^s}$ commutes with $\tau^s$ if and only if all $a_j$ are in $\F_{q^s}\cap \F_{q^\ell}=\F_{q^g}$. Moreover, $\alpha$ commutes with all elements of $\F_{q^\ell}$ if and only if 
	$$a_j\neq 0\quad  \Longrightarrow \quad \ell \mid js. 
	$$
	Since $\ell \mid js$ is equivalent to $n\mid js$, we conclude that $Z(\Delta(\ell, s))\subseteq K_g$. The reverse inclusion is clear. This proves (1). 
	
	We have $s\cdot \ell = n\cdot g$, so 
	$$\pi=\tau^n = (\tau^s)^{\ell/g}. 
	$$
	The powers $\tau^{sj}$, $0\leq j<\ell/g$, are distinct powers of $\tau$ with exponents less than $n$, so they are left linearly independent 
	over $k(\pi)$, and in particular over the subfield $\F_{q^\ell}(\pi)$ of $\Delta(\ell, s)$. Since $[\F_{q^\ell}(\pi):K]=\ell$, we get 
	$$
	\dim_K \Delta(\ell, s)\geq \ell\cdot \ell/g=\ell^2/g. 
	$$
	Similarly, $\dim_K \Delta(s, \ell)\geq s^2/g$. It is easy to see that $\Delta(s, \ell)\subseteq \Cent_\Delta(\Delta(\ell, s))$. By the Double Centralizer Theorem \cite[Cor. 7.13]{Reiner}, 
	$$
	\dim_K \Delta = \dim_K \Delta(\ell, s)\cdot \dim_K 	\Cent_\Delta(\Delta(\ell, s)). 
	$$
	Thus,
	$$
	n^2\geq ( \ell^2/g) ( s^2/g) = (\ell\cdot s/g)^2=n^2,
	$$
	so equalities must hold throughout. This proves (2) and (3). 
\end{proof}

\begin{prop}\label{prop3.2} There is $u\in \Delta(s, \ell)$ with $[K(u):K]=s$. For any such $u$, $L\colonequals K(u)$ is a maximal subfield
of $\Delta(s, \ell)$ and $\Cent_\Delta(u)=L\otimes_{K_g}\Delta(\ell,s)$. 
\end{prop}
\begin{proof} Since $K_g$ is the center of $\Delta(s, \ell)$ and $\dim_{K_g}\Delta(s, \ell)=(s/g)^2$, a maximal subfield $M$ of 
$\Delta(s,\ell)$ has $[M:K]=g\cdot s/g=s$. Since $K$ is a function field of transcendence degree $1$ over a finite field, the Primitive 
Element Theorem (cf. \cite[Thm. 1.5.19]{PapikianGTM}) gives $u$ with $M=K(u)$. Conversely, if $[K(u):K]=s$, then 
$K(u)\subseteq K_g(u)$ and $[K_g(u):K]\leq s$ force $K(u)=K_g(u)$ to be a maximal subfield. 
Because $K$ is the center of $\Delta$, we have $\Cent_\Delta(u)=\Cent_\Delta(L)$, so by the Double Centralizer Theorem $\Cent_\Delta(u)$ is a 
central division algebra over $L$ of dimension $(n/s)^2=(\ell/g)^2$. On the other hand, $L\otimes_{K_g}\Delta(\ell, s)\subseteq \Cent_\Delta(u)$, 
and the left side has dimension $(\ell/g)^2$ over $L$. Hence $L\otimes_{K_g}\Delta(\ell, s)=\Cent_\Delta(u)$. 
\end{proof}

\begin{rem}
	Note that in Proposition \ref{prop3.2} we can always scale $u$ by an element of $K$ to make it integral, i.e., to lie in $\twistt{\F_{q^s}}{\tau^\ell}$, 
	without changing $K(u)$. 
\end{rem}

\subsection{The construction of semifield codes}\label{sConstruction6} Recall that $\ell$ and $s$ are integers such that $n=\lcm(\ell, s)$, and we denoted 
$g=\gcd(\ell, s)$. 

Define $\phi\colon A\to \twist{k}$ by 
\begin{equation}\label{eqDefphi}
\phi_T=\sum_{i=0}^r a_i \tau^{\ell i}, \qquad a_i\in \F_{q^s}, \quad a_r\neq 0, 
\end{equation}
so that the image of $\phi$ lies in $\Delta(s, \ell)$ and $\phi$ has rank $r\ell$. The ring 
$$
S\colonequals \twistt{\F_{q^\ell}}{\tau^s}\subset \Delta(\ell, s)
$$
commutes with $\phi_T$, so $S\subseteq \End(\phi)$. We write $z=\tau^s$ and denote by $\deg_z(u)=\deg_\tau(u)/s$ the degree of $u\in S$ as a polynomial in $z$. 
Let $\fp\lhd A$ be a prime different from $\chr_A(\phi)$, let $d\colonequals \deg(\fp)$, and let 
\begin{equation}\label{eqcM}
\cM=\left\{\sum_{i=0}^{rd-1} b_i\tau^{s i}\mid b_i\in \F_{q^\ell}\right\}=\{u\in S\mid \deg_z(u)\leq rd-1\}. 
\end{equation}
We want to give conditions on $\phi$ and $\fp$ under which the image of $\cM$ in $\End_{\F_\fp}(\phi[\fp])\cong  \Mat_{r\ell}(\F_\fp)$ is a semifield code. 
Note that $\dim_{\F_q}\cM=rd\ell=r\ell\cdot \deg(\fp)$, as required by \eqref{eqSFdim}. 

The field $\F_{q^g}=\F_{q^\ell}\cap \F_{q^s}$ commutes with $\tau^s$, and $\pi=\tau^n=z^{\ell/g}$, so $\F_{q^g}[\pi]$ is contained in the center of $S$. 
In this subsection we apply the results of Section \ref{secDuality} over the field $\F_{q^g}$ instead of $\F_q$, i.e., with $q$ replaced by $q^g$ and 
$\tau$ replaced by $\tau^g$. The ring $\twistt{k}{\tau^g}$ is the twisted polynomial ring over $k$ with respect to the $q^g$-th power Frobenius; it 
contains $S$, because $g$ divides $s$, and $T\mapsto \pi=(\tau^g)^{n/g}$ defines a Drinfeld module $\F_{q^g}[T]\to \twistt{k}{\tau^g}$ of rank $n/g$, 
for which every element of $\twistt{k}{\tau^g}$ is an endomorphism. For $u\in S$ with $\deg_\tau(u)\geq 1$, put 
$$
d_u'\colonequals [\F_{q^g}(u, \pi):\F_{q^g}(u)]. 
$$

\begin{lem}\label{lem1.10M} Let $\tF\colonequals\F_q(\phi_T, \pi)=K(\phi_T)$.
\begin{enumerate}
\item If $u\in S$ and $\deg_\tau(u)\geq 1$, then
$$
d_u'=\deg_z(u)\cdot [K_g(u):K_g]\cdot g/\ell\leq \deg_z(u).
$$
\item $[\tF:F]=r\ell\cdot[\tF:K]/n\leq rg$, with equality if and only if $\tF$ is a maximal subfield of $\Delta(s, \ell)$.
\end{enumerate}
\end{lem}
\begin{proof}
(1) Proposition \ref{propDeg}, applied over $\F_{q^g}$ to $u$ and $\pi$, whose degrees in $\tau^g$ are $\deg_\tau(u)/g$ and $n/g$, gives 
$$
\frac{\deg_\tau(u)}{n}=\frac{[\F_{q^g}(u,\pi):\F_{q^g}(u)]}{[\F_{q^g}(u, \pi):K_g]}, 
$$
i.e., $d_u'=\deg_\tau(u)\cdot [K_g(u):K_g]/n$. Since $n=s\ell/g$, this is the claimed formula. The field $K_g(u)$ is a subfield of $\Delta(\ell, s)$ 
containing its center $K_g$, and $\dim_{K_g}\Delta(\ell, s)=(\ell/g)^2$, so $[K_g(u):K_g]\leq \ell/g$. 

(2) Proposition \ref{propDeg}, applied to $\phi_T$ and $\pi$, gives $r\ell/n=[\tF:F]/[\tF:K]$. The field $K_g(\phi_T)$ is a subfield of 
$\Delta(s, \ell)$ containing its center $K_g$, so $[K_g(\phi_T):K_g]\leq s/g$, and therefore $[\tF:K]\leq [K_g(\phi_T):K]\leq s$, with equality if and only if 
$\tF=K_g(\phi_T)$ is a maximal subfield of $\Delta(s,\ell)$. Since $r\ell s/n=rg$, this proves (2). 
\end{proof}

Let $\bar\pi$ be the image of $\pi$ in $\End_{\F_q}(\phi[\fp])$. The field $\F_{q^g}\subseteq S$ acts on $\phi[\fp]$ and commutes with $\bar\pi$, so 
$\F_{q^g}[\bar\pi]$ is a commutative subalgebra of $\End_{\F_q}(\phi[\fp])$. Let $G(x)\in \F_{q^g}[x]$ be the minimal polynomial of $\bar\pi$ over $\F_{q^g}$, 
so that $\F_{q^g}[\bar\pi]\cong \F_{q^g}[x]/(G(x))$. 

\begin{lem}\label{lemDegG}
$\deg(G)\leq rd$. 
\end{lem}
\begin{proof}
The torsion module $\phi[\fp]$ is a vector space of dimension $rd$ over $\F_{q^\ell}\subseteq S$, and $\bar\pi$ is $\F_{q^\ell}$-linear, since $\pi$ commutes with $k$. 
Let $\mu(x)\in \F_{q^\ell}[x]$ be the minimal polynomial of $\bar\pi$ over $\F_{q^\ell}$, so $\deg(\mu)\leq rd$. The image $\Theta$ of $\tau^s$ in 
$\End_{\F_q}(\phi[\fp])$ is invertible, commutes with $\bar\pi$, and satisfies $\Theta b=b^{q^s}\Theta$ for $b\in \F_{q^\ell}$. 
Hence $0=\Theta\mu(\bar\pi)\Theta^{-1}=\mu^{(q^s)}(\bar\pi)$, where $\mu^{(q^s)}$ is obtained by raising the coefficients of $\mu$ to the $q^s$-th power. 
By the uniqueness of the minimal polynomial, $\mu^{(q^s)}=\mu$, so the coefficients of $\mu$ lie in $\F_{q^\ell}\cap \F_{q^s}=\F_{q^g}$. Therefore $G=\mu$. 
\end{proof}

\begin{prop}\label{propCriterion}
Assume that $G$ is irreducible of degree $rd$, i.e., that $\F_{q^g}[\bar\pi]$ is a field with $q^{grd}$ elements. Then every nonzero $u\in \cM$ acts 
invertibly on $\phi[\fp]$, and the image of $\cM$ in $\End_{\F_\fp}(\phi[\fp])$ is a semifield code. 
\end{prop}
\begin{proof}
The nonzero elements of $\F_{q^\ell}$ act invertibly on $\phi[\fp]$. Let $u\in \cM$ with $\deg_\tau(u)\geq 1$, and suppose $u(\alpha)=0$ for some 
$0\neq \alpha\in \phi[\fp]$. Let $m_2'(T)$ be the minimal polynomial of $\pi$ over $\F_{q^g}(u)$, and let $\overline{m}_2'(T)\in \F_{q^g}[T]$ be its reduction 
modulo $u$, which is monic of degree $d_u'$. By Corollary \ref{corCLM}, applied over $\F_{q^g}$ to the Drinfeld module $T\mapsto \pi$ and its endomorphism $u$, 
we have $\overline{m}_2'(\pi)=wu$ for some $w\in \F_{q^g}[u, \pi]\subseteq S$. Hence $\overline{m}_2'(\bar\pi)$ is an element of the field $\F_{q^g}[\bar\pi]$ which 
kills $\alpha\neq 0$, so $\overline{m}_2'(\bar\pi)=0$, and $G$ divides $\overline{m}_2'$. By Lemma \ref{lem1.10M}(1), 
$$
rd=\deg(G)\leq d_u'\leq \deg_z(u)\leq rd-1, 
$$
a contradiction. The last claim follows from \eqref{eqSFdim}. 
\end{proof}


\subsection{The choice of $\phi$ and $\fp$}\label{sChoice} 
We now give conditions on $\phi$ and $\fp$, in terms of the Frobenius endomorphism of $\phi$, under which the assumption of 
Proposition \ref{propCriterion} holds, and we show that these conditions can be satisfied. 

Let $m_\phi(x)$ be the minimal polynomial of $\pi\in \End(\phi)$ over $F=\F_q(\phi_T)$. Then $m_\phi(x)$ 
is a monic irreducible polynomial in $A[x]$ of degree $[\tF:F]$. 
Let $\overline{m}_{\phi, \fp}(x)\in \F_\fp[x]$ be the polynomial obtained from $m_\phi(x)$ by reducing its coefficients modulo $\fp$. 

\begin{thm}\label{thmFamilies}
Let $\phi$ be the Drinfeld module of rank $r\ell$ defined by \eqref{eqDefphi}, let $\fp\neq \chr_A(\phi)$ be a prime of degree $d$, and let 
$\cM$ be defined by \eqref{eqcM}. Assume that $[\tF:F]=rg$ and that 
\begin{enumerate}
	\item $\overline{m}_{\phi, \fp}(x)$ is irreducible in $\F_\fp[x]$; 
	\item $\overline{m}_{\phi, \fp}(0)$ generates $\F_\fp$ over $\F_q$. 
\end{enumerate}
Then $\F_{q^g}[\bar\pi]$ is a field with $q^{grd}$ elements. Consequently, every nonzero element of $\cM$ acts invertibly on $\phi[\fp]$, and the image 
of $\cM$ in $\End_{\F_\fp}(\phi[\fp])\cong \Mat_{r\ell}(\F_\fp)$ is a semifield code. 
\end{thm}
\begin{proof} 
Since $m_\phi(\pi)=0$, the image of $A[\pi]$ in $\End_{\F_q}(\phi[\fp])$ is $\F_\fp[\bar\pi]$, a nonzero quotient of $\F_\fp[x]/(\overline{m}_{\phi, \fp}(x))$. 
The latter is a field by (1), so $\F_\fp[\bar\pi]$ is a field of degree $d\cdot [\tF:F]=rgd$ over $\F_q$. The conjugates of $\bar\pi$ over $\F_\fp$ are powers of $\bar\pi$, 
so $\Nr_{\F_\fp[\bar\pi]/\F_\fp}(\bar\pi)=\pm\overline{m}_{\phi, \fp}(0)$ lies in $\F_q[\bar\pi]$. By (2), $\F_\fp\subseteq \F_q[\bar\pi]$, so 
$\F_q[\bar\pi]=\F_\fp[\bar\pi]$ is a field with $q^{rgd}$ elements. 

The algebra $\F_{q^g}[\bar\pi]$ is generated by the commuting fields $\F_{q^g}$ and $\F_q[\bar\pi]$, so it is a quotient of 
$\F_{q^g}\otimes_{\F_q}\F_q[\bar\pi]$, which is a product of fields with $q^{rgd}$ elements. Hence $\F_{q^g}[\bar\pi]$ is a product of $j\geq 1$ such 
fields, and $\deg(G)=\dim_{\F_{q^g}}\F_{q^g}[\bar\pi]=j\cdot rd$. By Lemma \ref{lemDegG}, $j=1$. Now apply Proposition \ref{propCriterion}. 
\end{proof}

By Lemma \ref{lem1.10M}(2), 
\begin{equation}\label{eqtF1}
	[\tF:F]\leq rg, 
\end{equation}
so the assumption $[\tF:F]=rg$ of Theorem \ref{thmFamilies} says that this degree is as large as possible. 
By Lemma \ref{lem1.10M}(2), $[\tF:F]=rg$ holds if and only if $\tF=K(\phi_T)$ is a maximal subfield of $\Delta(s,\ell)$, and by
Proposition \ref{prop3.2} this is attained by suitable elements $\phi_T\in \Delta(s, \ell)$ (after scaling, integral elements; see the
remark following that proposition). 
However, Proposition \ref{prop3.2} gives no control over $\deg_\tau(\phi_T)$, i.e., over $r$. 
When $g=1$, the following remark shows that for a given $r$ the equality $[\tF:F]=r$ can be arranged 
by choosing $\phi$ ordinary. 

\begin{rem}\label{remOrdinary} Assume $g=1$. The Drinfeld module $\phi$ 
	defined by \eqref{eqDefphi} can be considered as an $\F_{q^\ell}[T]$-Drinfeld module $\phi'$ of rank $r$, with $\tau^\ell$ in place of $\tau$. 
	Let $\fq=\chr_A(\phi)$. Since $\deg(\fq)$ divides $s$, it is coprime to $\ell$, so $\fq$ remains prime in $\F_{q^\ell}[T]$ and 
	$H(\phi)=\ell H(\phi')$. In particular, $\ell\leq H(\phi)$. On the other hand, by Proposition \ref{lemCharMin} applied to $\phi$ 
	and $\pi$, $P_{\phi, \pi}(x)=m_\phi(x)^{r\ell/[\tF:F]}$, so by Proposition \ref{propHeight} the height $H(\phi)$ is a positive multiple 
	of $r\ell/[\tF:F]$. Thus, if $H(\phi)=\ell$, then $r\leq [\tF:F]$, and therefore $[\tF:F]=r$ by \eqref{eqtF1}. This means that, without appealing to 
	Proposition \ref{prop3.2}, we obtain $[\tF:F]=r$ by choosing $\phi'$ to be \textit{ordinary}, i.e., having $H(\phi')=1$. Such $\phi$ exist for every $r$: 
	for example, for $t\in \F_q$ one can take $\phi_T=t+\tau^\ell$ if $r=1$, and $\phi_T=t+\tau^\ell+\tau^{r\ell}$ if $r>1$. Then $\fq=T-t$ and 
	$\phi_\fq=\phi_T-t$ has height $\ell$ in $\tau$, so $H(\phi)=\ell$. 
\end{rem}

It remains to show that there are primes $\fp$ for which conditions (1) and (2) of Theorem \ref{thmFamilies} hold. 

\begin{lem}\label{lemCycle} Assume $g=1$ and $r$ is not divisible by the characteristic of $\F_q$. In addition, assume that either $r$ is prime, or 
that $r$ is coprime to $s$. Then the Galois group of $m_\phi(x)$ contains a cycle of maximal length $r$. 
\end{lem}
\begin{proof}
	Let $M$ be the splitting field of $m_\phi(x)$. Since $r=\deg m_\phi(x)$ is not divisible by the characteristic of $F$, the irreducible polynomial $m_\phi(x)$ has nonzero derivative, 
	so it is separable and the extension $M/F$ is Galois. 
	We claim that $\Gal(M/F)$, as a subgroup of the permutation group $S_r$ of the roots of $m_\phi(x)$, contains a cycle of length $r$. 
	We will show that the decomposition subgroup $G_\infty\subseteq \Gal(M/F)$ at the place $\infty=1/T$ contains such a cycle. 
	
	Since $\Delta$ has invariant $-1/n$ at the place $1/\pi$ of $K$, $\Delta\otimes_K K_{1/\pi}$ is a division algebra, so
	$\tF\otimes_K K_{1/\pi}\subseteq \Delta\otimes_K K_{1/\pi}$ is a field, i.e., $\tF$ has a unique place over $1/\pi$. Since $\pi$ is
	integral over $\F_q[\phi_T]$ and $\phi_T$ is integral over $\F_q[\pi]$ (see Section \ref{secDuality}), a place of $\tF$ is a pole of
	$\pi$ if and only if it is a pole of $\phi_T$. Hence $\tF$ also has a unique place over $\infty$, i.e., $m_\phi(x)$ remains
	irreducible over the completion $\Fi$ of $F$. Let 
	$M'$ be the splitting field of $m_\phi(x)$ over $\Fi$. Then $G_\infty\cong \Gal(M'/\Fi)$. 
	Let $\tinf$ be the unique place of $\tF$ over $\infty$, and let $\tF_{\tinf}$ be the completion of $\tF$ at $\tinf$. We have 
	$r=[\tF_{\tinf}:\Fi]=e(\tF_{\tinf}/\Fi) f(\tF_{\tinf}/\Fi)$, where $e(\tF_{\tinf}/\Fi) $ is the ramification index and $f(\tF_{\tinf}/\Fi)$ 
	is the residual degree. 
	
	If $r$ is a prime, then either  $e(\tF_{\tinf}/\Fi)=r$ or $f(\tF_{\tinf}/\Fi)=r$. In the second case, $M'=\tF_{\tinf}$ is an unramified extension of degree $r$, so $\Gal(M'/\Fi)\cong \Z/r\Z$; a generator of this group permutes the $r$ roots of the irreducible polynomial $m_\phi(x)$ transitively, hence is an $r$-cycle. In the first case, $\tF_{\tinf}/\Fi$ is a totally tamely ramified extension. 
	If $r$ is not necessarily prime, we compute the valuations of the roots of $m_\phi(x)$ at $\infty$. The residue fields of $F$ at
	$\infty$ and of $K$ at $1/\pi$ are both $\F_q$, so the residual degrees of $\tinf$ over these two places coincide, and therefore
	$e(\tinf/\infty)/e(\tinf/(1/\pi))=[\tF:F]/[\tF:K]=r\ell/n$ by Lemma \ref{lem1.10M}(2). Since $v_{\tinf}(\pi)=-e(\tinf/(1/\pi))$, the
	normalized valuation at $\infty$ of $\pi$, and hence of every root of $m_\phi(x)$, is $-n/r\ell=-s/r$. 
	Therefore, if $r$ is coprime to $s$, then again this implies that $\tF_{\tinf}/\Fi$ is a totally tamely ramified extension. 
	
	We are reduced to showing that if $\tF_{\tinf}/\Fi$ is a totally tamely ramified extension of degree $r$, 
	then $\Gal(M'/\Fi)$ contains a cycle of length $r$ in its action on the roots of $m_\phi(x)$. By a well-known structure theorem for totally tamely ramified extensions 
	of local fields (cf. \cite[Prop. 2.6.7]{PapikianGTM}), $M'$ is the splitting field of $x^r-\varpi_\infty$, where $\varpi_\infty$ 
	is a suitable uniformizer of $\Fi$; in particular, $M'$ contains a primitive $r$-th root $\zeta_r$ of $1$, and the inertia subgroup $I\subseteq \Gal(M'/\Fi)$ is cyclic of order $r$, generated by the automorphism $\sigma$ defined by 
	$\varpi_\infty^{1/r}\mapsto \zeta_r\varpi_\infty^{1/r}$. For any root $\alpha$ of $m_\phi(x)$, the extension $\Fi(\alpha)=\tF_{\tinf}$ of $\Fi$ 
	is totally ramified of degree $r=\# I$, so $I\cap \Gal(M'/\Fi(\alpha))=1$. Thus $I=\langle \sigma\rangle$ acts freely on the $r$ roots 
	of $m_\phi(x)$, so $\sigma$ is an $r$-cycle. 
\end{proof}

Assume the Galois group of $m_\phi(x)$ contains a cycle of maximal length $rg$; by the previous lemma, this holds when $g=1$ and $r$ satisfies the assumptions stated there. (For $g>1$ we do not know a general criterion guaranteeing the existence of such a cycle, so the existence statement below,
Theorem \ref{thmExistence}, is restricted to $g=1$; Theorem \ref{thmFamilies} holds for all $g$.)
Let $M$ be the splitting field of $m_\phi(x)$ over $F$, let $\F_{q^h}$ be the algebraic closure of $\F_q$ in $M$, and let 
$\gamma\in \Gal(M/F)$ be an $rg$-cycle. The restriction of $\gamma$ to $\F_{q^h}F$ is a power $\mathrm{Frob}_q^{N_0}$ of the 
$q$-Frobenius, $N_0\in \Z/h\Z$. By the effective Chebotarev density theorem for function fields (cf. \cite{MuSch}), 
for $N\equiv N_0 \pmod h$ sufficiently large, the number of primes in $A$ of degree $N$ which are unramified in $M$, do not divide 
the discriminant of $m_\phi(x)$, and whose Frobenius is in the conjugacy class of $\gamma$ is $\geq cq^N/N$ for some constant $c>0$ not depending on $N$. 
Denote the set of such primes by $\cP_N$. 

For $\fp\in \cP_N$ the reduction of $m_\phi(x)$ modulo $\fp$ is irreducible (i.e., (1) holds), so we now concentrate on condition (2).  
Let $\fq=\chr_A(\phi)$. By Proposition \ref{lemCharMin} (applied to $\phi$ and $\pi$), $P_\phi(x)=m_\phi(x)^{r\ell/rg}=m_\phi(x)^{\ell/g}$. Let $U$ be the matrix of $\pi$ acting on $M_\phi$, 
so that $\det(U)=\pm P_\phi(0)$ has degree $n$ in $T$ by Lemma \ref{lemDegDet}.
Since $\pi=\tau^n$ is injective on $\bar{k}$, it acts invertibly on $\phi[\fp]$ for every prime $\fp\neq \fq$, so by Corollary
\ref{corInv} no such $\fp$ divides $\det(U)$. Hence, up to $\F_q^\times$-multiples, $P_\phi(0)=\fq^{n/\deg(\fq)}$ and
$m_\phi(0)=\fq^{s/\deg(\fq)}$. 

Denote $f(T)=\fq^{s/\deg(\fq)}\in A$. Note that $\deg_T f=s$. Choose a root of each $\fp\in \cP_N$ in $\overline{\F}_q$ and call the set of these roots $\cR_N$. We have 
$\# \cR_N\geq  cq^N/N$. We want to show that $f(T)$ modulo $\fp$ generates $\F_\fp=\F_{q^N}$ over $\F_q$ for at least one $\fp\in \cP_N$. 
This is equivalent to requiring that for some $\alpha\in \cR_N$, the value $f(\alpha)$ 
does not belong to a proper subfield of $\F_{q^N}$.
The image of the map $\cR_N\to \F_{q^N}$, $\alpha\mapsto f(\alpha)$, has size at least $\# \cR_N/\deg(f)$ since the preimages of this map for $\beta\in \F_{q^N}$ 
are the roots of $f(x)-\beta$, and there at most $\deg(f)$ such roots. Finally, note that the union of proper subfields of $\F_{q^N}$ has size at most $O(q^{N/2})$. 
Since 
$$
c\frac{q^N}{N\cdot s} > O(q^{N/2})
$$ 
when $N\equiv N_0 \pmod h$ is sufficiently large, we see that we can always find $\fp$ such that the properties (1) and (2) are satisfied. 

 We have proved: 

\begin{thm}\label{thmExistence}
Assume that $g=1$, that $r$ is not divisible by the characteristic of $\F_q$, and that $r$ is prime or coprime to $s$. Let $\phi$ be 
defined by \eqref{eqDefphi} with $H(\phi)=\ell$; such $\phi$ exist by Remark \ref{remOrdinary}. Then $[\tF:F]=r$, and there are infinitely 
many primes $\fp$ satisfying the assumptions of Theorem \ref{thmFamilies}. 
\end{thm}

Indeed, $[\tF:F]=r$ by Remark \ref{remOrdinary}, the Galois group of $m_\phi(x)$ contains an $r$-cycle by Lemma \ref{lemCycle}, and the 
argument above produces a prime satisfying (1) and (2) of every sufficiently large degree $N\equiv N_0\pmod h$. 

 \subsection{Cyclic semifields and the nuclear parameters of the code}\label{ssPetit} 
Let $\phi$, $\fp$, $S$ and $\cM$ be as in $\S$\ref{sConstruction6}. In this subsection we assume that all nonzero elements of 
$\cM$ act invertibly on $\phi[\fp]$, so that $\cM$ gives a semifield code in $\End_{\F_\fp}(\phi[\fp])$. By Proposition \ref{propCriterion}, this holds 
if the minimal polynomial $G$ of $\bar\pi$ over $\F_{q^g}$ is irreducible of degree $rd$. We show that the corresponding semifield is a cyclic semifield, 
and we determine the nuclear parameters of $\cM$ in the sense of Definition \ref{defNPq}. 

Let $\sigma$ be the restriction of $\Frob_q^s$ to $\F_{q^\ell}$. It generates $\Gal(\F_{q^\ell}/\F_{q^g})$, so it has order $\ell/g$. 
Recall that $z=\tau^s$. Since $zb=\sigma(b)z$ for all $b\in \F_{q^\ell}$, the ring $S=\twistt{\F_{q^\ell}}{\tau^s}$ 
is the skew polynomial ring $\F_{q^\ell}[z; \sigma]$, and $\cM=\{f\in S\mid \deg_z(f)\leq rd-1\}$. The center of $S$ is $\F_{q^g}[z^{\ell/g}]=\F_{q^g}[\pi]$. 
The ring $S$ has a right division algorithm. As in Example \ref{exPetit}, which is the case $\ell=n$ and $s=1$, a monic irreducible $P\in S$ of degree 
$m$ in $z$ defines the cyclic semifield 
$$
\cS_P=\{f\in S\mid \deg_z(f)\leq m-1\}, \qquad f\circ h=fh\modr P,
$$
of order $q^{\ell m}$. The semifield code of $\cS_P$ is the set of the maps $h\mapsto f\circ h$, $f\in \cS_P$. 

\begin{prop}\label{propPetit}
There is a monic irreducible $P\in S$ of degree $rd$ in $z$ and an isomorphism of left $S$-modules $S/SP\xrightarrow{\ \sim\ } \phi[\fp]$ under which 
the action of $f\in \cM$ on $\phi[\fp]$ corresponds to the map $h\mapsto fh\modr P$ on $\cS_P$. Thus, the semifield code defined by $\cM$ is the 
semifield code of the cyclic semifield $\cS_P$, written in a different basis. 
\end{prop}
\begin{proof}
Since $S\subseteq \End(\phi)$, the torsion module $W\colonequals \phi[\fp]$ is a left $S$-module. It has dimension $rd$ over $\F_{q^\ell}$, because 
$\dim_{\F_q}W=r\ell d$. Fix $0\neq v\in W$ and consider the homomorphism of left $S$-modules $\ev\colon S\to W$, $f\mapsto f(v)$. 
Its restriction to $\cM$ is injective, because the nonzero elements of $\cM$ act invertibly on $W$, hence bijective, because 
$\dim_{\F_q}\cM=rd\ell=\dim_{\F_q}W$. The kernel of $\ev$ is a left ideal of $S$, so it is equal to $SP$ for a unique monic $P\in S$, and 
$\ev$ induces an isomorphism $S/SP\cong W$. Comparing the dimensions over $\F_{q^\ell}$, we get $\deg_z(P)=rd$. 
If $P=P_1P_2$ with $P_1, P_2\in S$ of positive degree, then $P_1$ and $P_2$ are nonzero elements of $\cM$, so $P_2(v)\neq 0$ and 
$P_1(P_2(v))=P(v)=0$, which contradicts the invertibility of $P_1$ on $W$. Hence $P$ is irreducible. Finally, for $f, h\in \cM$ we have 
$f(h(v))=(fh)(v)=(fh\modr P)(v)$, which proves the last claim. 
\end{proof}

\begin{cor}\label{corNP6}
Assume $rd\geq 2$ and $\ell\nmid s$. Then 
\[ 
	I_l(\cM)=I_r(\cM)=\F_{q^\ell}, \qquad Z(\cM)=\F_{q^g}, 
\]
and $C(\cM)$ is the image of $\F_{q^g}[\pi]$ in $\End_{\F_\fp}(\phi[\fp])$, which is a field with $q^{rdg}$ elements. Thus, the nuclear parameters of $\cM$ are 
\[
(rd\ell;\ \ell, \ell, rdg;\ g).
\]
\end{cor}
\begin{proof}
The assumption $\ell\nmid s$ means that $\sigma\neq \id$, and $\deg_z(P)=rd\geq 2$. In this case the nuclei of the cyclic semifield $\cS_P$ and the 
idealizers and the centralizer of its semifield code are known; see \cite[\S 3.1]{Sheekey} and \cite[\S 2.4]{SZ}. The left and middle nuclei are 
isomorphic to $\F_{q^\ell}$, the center is $\F_{q^g}$, and the centralizer of the semifield code is the image of the center $\F_{q^g}[z^{\ell/g}]$ of $S$ in 
$\End_{\F_q}(S/SP)$, which is a field with $(q^g)^{rd}$ elements. By Proposition \ref{propPetit}, the same holds for $\cM$ acting on $\phi[\fp]$, with 
$z^{\ell/g}=\pi$. Since $\F_{q^\ell}\subseteq I_l(\cM)\cap I_r(\cM)$ and $\F_{q^g}\subseteq Z(\cM)$, the claim follows. 
\end{proof}

\begin{rem}\label{remSZ}
By \cite{SZ}, the cyclic semifield $\cS_P$ is determined up to isotopy by the minimal polynomial over $\F_{q^g}$ of the central element 
$z^{\ell/g}$ acting on $S/SP$, which is irreducible of degree $\deg_z(P)$. In our setting $z^{\ell/g}=\pi$, and by Proposition \ref{propPetit} 
this polynomial is the minimal polynomial $G$ of $\bar\pi$ over $\F_{q^g}$. Hence the condition of Proposition \ref{propCriterion} is also 
necessary for $\cM$ to give a semifield code, and $G$ determines the isotopy class of the corresponding semifield. 
\end{rem}

\begin{example} Let $q=3$, $\ell=2$ and $s=3$, so $n=6$ and $g=1$. Let $\alpha$ be a root of $x^3-x+1$, and let 
	$\phi_T=\alpha+\alpha^2\tau^2+\tau^4$, so $r=2$. The minimal polynomial of $\pi$ over $\F_q(\phi_T)$ is $m_\phi(x)=x^2-Tx-(T^3-T+1)$, 
	so $[\tF:F]=2=rg$, and its reduction $x^2-x-1$ modulo $\fp=T-1$ is irreducible over $\F_q$. Since $d=1$, the assumptions of Theorem \ref{thmFamilies} 
	hold, so $\cM=\{b_0+b_1\tau^3\mid b_0, b_1\in \F_{q^2}\}$ gives a semifield code in $\End_{\F_\fp}(\phi[\fp])\cong \Mat_4(\F_q)$ 
	(this can also be checked directly with the division algorithm in $\twist{k}$). By Corollary \ref{corNP6}, its nuclear parameters are $(4; 2, 2, 2; 1)$. 
\end{example}

 

\section*{Declaration of generative AI and AI-assisted technologies in the manuscript preparation process}

During the preparation of this work the authors used Claude (Anthropic) and ChatGPT (OpenAI) as follows. 
A search of the literature on nuclear parameters of semifields, carried out with the help of these tools, identified nuclear parameters 
not covered by known constructions, among them $(9; 3, 1, 3; 1)$ for semifields of order $q^9$. 
The tools were then used to write computer code which, for $q\in\{2, 3\}$, found new semifield codes of the form \eqref{eqcMU}. 
Based on these examples, the authors determined a strategy to construct a general family. 
The tools were also used to assist in carrying out this strategy, in developing some of the proofs, and in polishing the exposition. 
After using these tools, the authors reviewed and edited the content as needed and take full responsibility for the content and 
correctness of the paper.

\bibliographystyle{amsalpha}
\bibliography{bibliography}

\end{document}